\documentclass[preprint,aos]{imsart}

\usepackage{amsthm,amsmath}
\usepackage{amssymb}

\startlocaldefs

\usepackage{etex}
\usepackage{xspace,enumerate}
\usepackage[dvipsnames]{xcolor}
\usepackage[T1]{fontenc}
\usepackage[full]{textcomp}
\usepackage{amsthm}
\newtheorem{theorem}{Theorem}[section]
\newtheorem*{theorem*}{Theorem}

\newtheorem*{proposition*}{Proposition}
\newtheorem{lemma}[theorem]{Lemma}
\newtheorem*{lemma*}{Lemma}

\newtheorem*{conjecture*}{Conjecture}
\newtheorem{fact}[theorem]{Fact}
\newtheorem*{fact*}{Fact}

\newtheorem*{hypothesis*}{Hypothesis}

\theoremstyle{definition}
\newtheorem{definition}[theorem]{Definition}
\newtheorem*{definition*}{Definition}

\newtheorem{example}[theorem]{Example}

\newtheorem{problem}[theorem]{Problem}

\theoremstyle{remark}

\newtheorem*{claim*}{Claim}

\newtheorem*{remark*}{Remark}

\newtheorem*{observation*}{Observation}
\usepackage{bm} 
\usepackage{microtype}
\usepackage[
pagebackref,
colorlinks=true,
urlcolor=blue,
linkcolor=blue,
citecolor=OliveGreen,
]{hyperref}
\usepackage[capitalise,nameinlink]{cleveref}
\crefname{lemma}{Lemma}{Lemmas}
\crefname{fact}{Fact}{Facts}
\crefname{theorem}{Theorem}{Theorems}
\crefname{corollary}{Corollary}{Corollaries}
\crefname{claim}{Claim}{Claims}
\crefname{example}{Example}{Examples}
\crefname{problem}{Problem}{Problems}
\crefname{definition}{Definition}{Definitions}
\usepackage{paralist}
\usepackage{turnstile}
\usepackage{mdframed}
\usepackage{algorithm}
\usepackage{algpseudocode}
\usepackage{tikz}

\newcommand{\proves}[1]{\vdash_{#1}}

\usepackage{boxedminipage}

\newcommand{\Paren}[1]{\left(#1\right)}

\newcommand{\Brac}[1]{\left[#1\right]}

\newcommand{\Abs}[1]{\left\lvert#1\right\rvert}

\newcommand{\Norm}[1]{\left\lVert#1\right\rVert}

\newcommand{\iprod}[1]{\langle#1\rangle}

\newcommand{\Esymb}{\mathbb{E}}
\newcommand{\Psymb}{\mathbb{P}}

\DeclareMathOperator*{\E}{\Esymb}

\DeclareMathOperator*{\ProbOp}{\Psymb}
\renewcommand{\Pr}{\ProbOp}

\newcommand{\mper}{\,.}
\newcommand{\mcom}{\,,}
\newcommand\bdot\bullet

\DeclareMathOperator{\Tr}{Tr}

\DeclareMathOperator{\poly}{poly}

\newcommand{\N}{\mathbb N}
\newcommand{\R}{\mathbb R}

\newcommand{\cA}{\mathcal A}

\newcommand{\cF}{\mathcal F}

\newcommand{\cN}{\mathcal N}

\newcommand{\cS}{\mathcal S}

\newcommand{\bZ}{\mathbf Z}

\renewcommand{\leq}{\leqslant}

\renewcommand{\geq}{\geqslant}

\let\epsilon=\varepsilon
\numberwithin{equation}{section}
\newcommand\MYcurrentlabel{xxx}
\newcommand{\MYstore}[2]{%
  \global\expandafter \def \csname MYMEMORY #1 \endcsname{#2}%
}
\newcommand{\MYload}[1]{%
  \csname MYMEMORY #1 \endcsname%
}
\newcommand{\MYnewlabel}[1]{%
  \renewcommand\MYcurrentlabel{#1}%
  \MYoldlabel{#1}%
}
\newcommand{\MYdummylabel}[1]{}
\newcommand{\torestate}[1]{%
  \let\MYoldlabel\label%
  \let\label\MYnewlabel%
  #1%
  \MYstore{\MYcurrentlabel}{#1}%
  \let\label\MYoldlabel%
}
\newcommand{\restatetheorem}[1]{%
  \let\MYoldlabel\label
  \let\label\MYdummylabel
  \begin{theorem*}[Restatement of \cref{#1}]
    \MYload{#1}
  \end{theorem*}
  \let\label\MYoldlabel
}
\newcommand{\restatelemma}[1]{%
  \let\MYoldlabel\label
  \let\label\MYdummylabel
  \begin{lemma*}[Restatement of \cref{#1}]
    \MYload{#1}
  \end{lemma*}
  \let\label\MYoldlabel
}
\newcommand{\restateprop}[1]{%
  \let\MYoldlabel\label
  \let\label\MYdummylabel
  \begin{proposition*}[Restatement of \cref{#1}]
    \MYload{#1}
  \end{proposition*}
  \let\label\MYoldlabel
}
\newcommand{\restatefact}[1]{%
  \let\MYoldlabel\label
  \let\label\MYdummylabel
  \begin{fact*}[Restatement of \prettyref{#1}]
    \MYload{#1}
  \end{fact*}
  \let\label\MYoldlabel
}
\newcommand{\restate}[1]{%
  \let\MYoldlabel\label
  \let\label\MYdummylabel
  \MYload{#1}
  \let\label\MYoldlabel
}

\newcommand{\e}{\epsilon}

\allowdisplaybreaks
\newcommand*{\Id}{\mathrm{Id}}

\newcommand*{\loweredwidetildehelper}[2]{\hbox{\csname dimen@\endcsname\accentfontxheight#1%
  \accentfontxheight#11.25\csname dimen@\endcsname
  $\csname m@th\endcsname#1\widetilde{#2}$%
  \accentfontxheight#1\csname dimen@\endcsname
  }%
}
\newcommand*{\accentfontxheight}[1]{\fontdimen5\ifx#1\displaystyle \textfont \else\ifx#1\textstyle \textfont \else\ifx#1\scriptstyle \scriptfont \else \scriptscriptfont \fi\fi\fi3
}

\DeclareMathOperator{\pEE}{\tilde{\mathbb{E}}}
\newcommand{\pE}{\pEE\nolimits}

\endlocaldefs

\begin{document}

\begin{frontmatter}

\title{Mean estimation with sub-Gaussian rates in polynomial time}
\runtitle{Polynomial-time mean estimation}


\begin{aug}
\author{\fnms{Samuel B.} \snm{Hopkins} \thanksref{t1} \thanksref{t2}}
\thankstext{t1}{The author is supported by a Miller Fellowship at UC Berkeley. This project was also supported by NSF Award No. 1408673.}
\thankstext{t2}{MSC 2010 classification: Primary 62H12, Secondary 68W. Keywords and phrases: multivariate estimation, heavy tails, confidence intervals, sub-Gaussian rates, semidefinite programming, sum of squares method.}
\address{\url{hopkins@berkeley.edu} \\ \url{www.samuelbhopkins.com} \\ Soda Hall \\ Berkeley, CA 94720} 
\affiliation{Department of Electrical Engineering and Computer Science \\ University of California, Berkeley}

\runauthor{Samuel B. Hopkins}
\end{aug}

\begin{abstract}
We study polynomial time algorithms for estimating the mean of a heavy-tailed multivariate random vector.
We assume only that the random vector $X$ has finite mean and covariance.
In this setting, the radius of confidence intervals achieved by the empirical mean are large compared to the case that $X$ is Gaussian or sub-Gaussian.

We offer the first polynomial time algorithm to estimate the mean with sub-Gaussian-size confidence intervals under such mild assumptions.
Our algorithm is based on a new semidefinite programming relaxation of a high-dimensional median.
Previous estimators which assumed only existence of finitely-many moments of $X$ either sacrifice sub-Gaussian performance or are only known to be computable via brute-force search procedures requiring time exponential in the dimension.

\end{abstract}



\end{frontmatter}

\section{Introduction}
\label[section]{sec:intro}
This paper studies estimation of the mean of a heavy-tailed multivariate random vector from independent samples.
In particular, we address the question: \emph{Are statistically-optimal confidence intervals for heavy-tailed multivariate mean estimation achievable by polynomial-time computable estimators?}
Our main result answers this question affirmatively, up to some explicit constants.

Estimating the mean of a distribution from independent samples is among the oldest problems in statistics.
From the \emph{asymptotic} viewpoint (that is, when the number of samples $n$ tends to infinity) it is well understood.
If $X_1,\ldots,X_n$ are $n$ independent copies of a random variable $X$ on $\R^d$, the empirical mean $\overline{\mu}_n = \tfrac 1n \sum_{i \leq n} X_i$ converges in probability to the mean $\mu = \E X$.
If $X$ has finite variance, the limiting distribution of $\overline{\mu}_n$ is Gaussian.

Aiming for finer-grained (finite-sample) guarantees, this paper takes a \emph{non-asymptotic} view.
For every $\delta > 0$ and $n \in \N$ we ask for an estimator $\hat{\mu}_{n,\delta}$ which comes with a tail bound of the form
\[
  \Pr_{X_1,\ldots,X_n} \left \{ \Norm{\hat{\mu}_{n,\delta}(X_1,\ldots,X_n) - \mu} > r_\delta \right \} \leq \delta
\]
for as small a radius $r_\delta$ (which may depend on $n$ and the distribution of $X$) as possible.
That is, we are interested in estimators with the smallest-possible confidence intervals.

When $X$ is Gaussian or sub-Gaussian, strong non-asymptotic guarantees are available on confidence intervals of the sample mean $\overline{\mu}_n$.
Applying Gaussian concentration, if $X$ has covariance $\Sigma$, then in the Gaussian setting,
\begin{align}\label[equation]{eq:gauss-concentration}
  \Pr \left \{ \Norm{\overline{\mu}_{n}(X_1,\ldots,X_n) - \mu} > \sqrt{\frac{ \Tr \Sigma}{n}} + \sqrt{\frac{2 \|\Sigma\| \log(1/\delta)}{n}} \right \} \leq \delta
\end{align}
where $\|\Sigma\| = \lambda_{\text{max}}(\Sigma)$ is the operator norm/maximum eigenvalue of $\Sigma$.

However, if one tries to replace the assumption that $X$ is Gaussian with something weaker, \cref{eq:gauss-concentration} breaks down for the sample mean $\overline{\mu}_n$.
For instance, consider a much weaker assumption: $X$ has finite covariance $\Sigma$.
Then the best possible tail inequality for the sample mean becomes
\begin{align}\label[equation]{eq:bad-concentration}
 \Pr \left \{ \Norm{\overline{\mu}_{n}(X_1,\ldots,X_n) - \mu} > \sqrt{\frac{ \Tr \Sigma}{\delta n}} \right \} \leq \delta\mper
\end{align}
(See e.g. \cite{catoni2012challenging}, section 6.)
By comparison with \cref{eq:gauss-concentration}, the tail bound \cref{eq:bad-concentration} has degraded in two ways: first, the $\log(1/\delta)$ term has become $1/\delta$, and second, that term multiplies $\Tr \Sigma$ rather than $\|\Sigma\|$; note that $\Tr \Sigma$ may be as large as $d \|\Sigma\|$, as in the case of isotropically-distributed data.

This paper focuses on finding estimators $\hat{\mu}$ which can match \eqref{eq:gauss-concentration} under milder assumptions than sub-Gaussianity, such as the existence of finitely-many moments.
Weak assumptions like this allow for the presence of \emph{heavy tails}.
A $d$-dimensional random vector $X$ is heavy-tailed if for some unit $u \in \R^d$, the tail of $\iprod{X,u}$ outgrows any exponential distribution; i.e. for all $s > 0$ one has $\lim_{t \rightarrow \infty} e^{ts} \Pr\{ \iprod{X - \mu,u} > t\} = \infty$.

There are many situations in which one may wish to avoid a Gaussian or sub-Gaussian assumption.
One may simply wish to be conservative, or there may reason to believe a Gaussian assumption is unjustified -- heavy-tailed and high-dimensional data are not unusual.
Many distibutions in big-data settings have heavy tails: for example, \emph{power law} distributions consistently emerge from statistics of large networks (the internet graph, social network graphs, etc) \cite{faloutsos1999power, leskovec2005graphs}.
And no matter how nice the underlying distribution, corruptions and noise in collected data often result in an empirical distribution with many outliers \cite{rahm2000data}.
As a result, such $X$ may have only a few finite moments; that is, $\E X^p$ may not exist for large-enough $p \in \N$.

This suggests the question of whether an estimator with a guarantee matching \cref{eq:gauss-concentration} (up to universal constants) exists under only the assumption that $X$ has finite mean and covariance.
(These assumptions are necessary to obtain the $1/\sqrt{n}$ rate in both \cref{eq:gauss-concentration,eq:bad-concentration}.)
One may show this is impossible if a single estimator is desired to satisfy an inequality like \cref{eq:gauss-concentration} \cite{devroye2016sub}.

Quite remarkably, the story changes if the estimator may additionally depend on the desired confidence level $1-\delta$.
Indeed, by now in the classical case $d = 1$, many such \emph{$\delta$-dependent} estimators are known which achieve \cref{eq:gauss-concentration} up to explicit constants for $\delta \geq 2^{-O(n)}$, even when $X$ has only finite mean and variance \cite{catoni2012challenging, devroye2016sub}.
Since the $\delta$-dependence is a necessary concession to achieve concentration like \cref{eq:gauss-concentration} with only two finite moments, for this paper our estimators are all allowed to depend on $\delta$: it is an interesting future direction to explore what fraction of the theory may be reproduced without the $\delta$-dependence \cite{devroye2016sub, minsker2018uniform}.
The lower bound $\delta \geq 2^{-O(n)}$ is also information-theoretically necessary \cite{devroye2016sub}.

The high-dimensional case is much more difficult, and has been resolved only recently: the culmination of a series of works \cite{lerasle2011robust,hsu2016loss,minsker2015geometric,LM18} is the following theorem of Lugosi and Mendelson, who gave the family of estimators matching \cref{eq:gauss-concentration} (up to constants) for any $d$ under only the assumption of finite second moments.
(In fact, their result also holds in the infinite-dimensional Banach space setting.)
\begin{theorem}[Lugosi-Mendelson estimator, \cite{LM18}]\label[theorem]{thm:lugosi-mendelson}
  There is a universal constant $C$ such that for every $n,d$, and $\delta \geq 2^{-n/C}$ there is an estimator $\hat{\mu}_{\delta,n} \, : \, \R^{dn} \rightarrow \R^d$ such that for every random variable $X$ on $\R^d$ with finite mean and covariance,
  \[
  \Pr \left \{ \Norm{\hat{\mu}_{n,\delta}(X_1,\ldots,X_n) - \mu} > C \Paren{ \sqrt{\frac{ \Tr \Sigma}{n}} +  \sqrt{\frac{ \|\Sigma\| \log(1/\delta)}{n}}} \right \} \leq \delta
  \]
  where $X_1,\ldots,X_n$ are i.i.d. copies of $X$ and $\mu = \E X$ and $\Sigma = \E (X - \mu)(X - \mu)^\top$.
\end{theorem}

In high-dimensional estimation, especially with large data sets, it is important to study estimators with guarantees both on statistical accuracy and algorithmic tractability.
Indeed, there is growing evidence that some basic high-dimensional estimation tasks which appear possible from a purely information-theoretic perspective altogether lack computationally efficient algorithms.
There are many examples of such \emph{information-computation gaps}, including the problem of finding sparse principal components of high-dimensional data sets (the \emph{sparse PCA problem}) and optimal detection of hidden communities in random graphs with latent community structure (the \emph{$k$-community stochastic block model}) \cite{DBLP:conf/colt/BerthetR13, DBLP:conf/nips/MaW15, HKKMP17, decelle2011inference, BKM17, HS17}.

From this perspective, a major question left open by \cref{thm:lugosi-mendelson} is whether there exists an estimator matching \cref{thm:lugosi-mendelson} but which is efficiently computable.
In this paper, \emph{efficiently computable} means computable by an algorithm running in time $(nd \log(1/\delta) )^{O(1)}$ -- that is, polynomial in both the number of samples and the ambient dimension, as well as the number of bits needed to describe the input $\delta > 0$.
Indeed, the \emph{median-of-means} estimator used by Lugosi and Mendelson lacks any obvious algorithm running in time less than $\exp(cd)$, for some fixed $c > 0$, which is the time required for brute-force search over every direction in a $d$-dimensional $\e$-net.
More worringly, the key idea of Lugosi and Mendelson is a combinatorial notion of a multivariate median, which appears to place the problem dangerously near those high-dimensional combinatorial statistics problems which lack efficient algorithms altogether.

The main result of this paper shows that there is a family of estimators matching \cref{thm:lugosi-mendelson} and computable by polynomial-time algorithms.

\begin{theorem}[Main theorem]\label[theorem]{thm:main}
  There are universal constants $C_0,C_1,C_2$ such that for every $n,d \in \N$ and $\delta > 2^{-n/C_2}$ there is an algorithm which runs in time $O(nd) + (d \log(1/\delta))^{C_0}$ such that for every random variable $X$ on $\R^d$, given i.i.d. copies $X_1,\ldots,X_n$ of $X$ the algorithm outputs a vector $\hat{\mu}_\delta(X_1,\ldots,X_n)$ such that
\[
  \Pr \left \{ \Norm{\mu - \hat{\mu}_\delta } > C_1 \Paren{\sqrt{\frac{\Tr \Sigma}{n}} + \sqrt{\frac{ \|\Sigma\| \log(1/\delta)}{n}}} \right \} \leq \delta\mcom
\]
where $\E X = \mu$ and $\E(X - \mu)(X - \mu)^\top = \Sigma$.
\end{theorem}

\paragraph{On constants and running times}
No effort has been made to optimize the constants $C_0,C_1,C_2$.
By careful analysis they may certainly be made less than $1000$, but we expect substantial improvements beyond this are possible.

Because of the large polynomial running time, we regard \cref{thm:main} as mainly \emph{a (constructive) proof of the existence of a polynomial-time algorithm:} of course we do not suggest anyone attempt to run an $(nd)^{1000}$-time algorithm in practice!
Polynomial-time algorithms are qualitatively different from exponential-time brute-force searches, however, and very often the insights from a slow polynomial-time algorithm can be leveraged to design a fast one, while the same cannot be said of a brute-force search procedure.
Thus, when addressing challenging algorithmic questions in high-dimensional statistics, the first question is whether there is a polynomial-time algorithm at all: \cref{thm:main} answers this affirmatively.

Indeed, \cref{thm:main} and the algorithm behind it have already inspired further investigation into the (rather distinct) question of just how fast an algorithm is possible.
After the present work was initially circulated, Cherapanamjeri, Flammarion, and Bartlett combined the ideas in our \cref{sec:cert-central} with a nonconvex gradient descent procedure to obtain an algorithm with the statistical same guarantees as \cref{thm:main} but with running time $O(n^{3.5} + n^2 d) \cdot (\log nd)^{O(1)}$ \cite{cherapanamjeri2019fast}.
It is more than plausible that further developments will lead to a truly practical algorithm (with running time, say, $nd \cdot \log(nd)^{O(1)}$ -- note that input vectors consist of $nd$ real numbers, so this running time would correspond to reading the data $\log(nd)^{O(1)}$ times).

\paragraph{Semidefinite programming, proofs to algorithms, and the sum of squares method}
Our algorithm is based on semidefinite programming (SDP).
It is not an attempt to directly compute the estimator proposed by Lugosi and Mendelson.
Instead, inspired by that estimator, we introduce \textsc{median-sdp}, a new semidefinite programming approach to computation of a high-dimensional median.
We hope that the ideas behind it will find further uses in algorithms for high-dimensional statistics.

Our SDP arises from the \emph{sum of squares (SoS)} method, which is a powerful and flexible approach to SDP design and analysis.
Rather than design an SDP from scratch and invent a new analysis, guided by the SoS method we construct an SDP whose variables and constraints allow for the \emph{proof} of Lugosi and Mendelson's \cref{thm:lugosi-mendelson} to translate directly to an \emph{analysis} of the SDP, proving our \cref{thm:main}.
(More prosaically: Lugosi and Mendelson's proof inspires the construction of a family of dual solutions to our SDP, which then we use to argue that it recovers a good estimate for the mean.)

This technique, which turns sufficiently-simple proofs of identifiability like the proof of \cref{thm:lugosi-mendelson} into algorithms as in \cref{thm:main}, has recently been employed in algorithm design for several computationally-challenging statistics problems.
For instance, recent works offer the best available polynomial-time guarantees for parameter estimation of high-dimensional mixture models and for estimation in Huber's contamination model \cite{huber1964robust, hopkins2018mixture, kothari2018robust, klivans2018efficient}.
SoS has also been key to progress in computationally-challenging tensor problems with statistical applications, such as tensor decomposition (a key primitive for moment-method algorithms in high dimensions) and tensor completion \cite{DBLP:conf/focs/MaSS16, DBLP:conf/colt/BarakM16, potechin2017exact}.
For further discussion see the survey \cite{RSS18}.
We expect many further basic statistical problems for which efficient algorithms are presently unknown to be successfully attackable with the SoS method.

\paragraph{Organization}
In the remainder of this introduction we discuss the \emph{median of means} estimation paradigm which underlies both Lugosi and Mendelson's estimator (\cref{thm:lugosi-mendelson}) and our own (\cref{thm:main}) and briefly introduce the SoS method, as well as offer some comparisons of the SDP used in this paper to some common SDPs employed in statistics.
Before turning to technical material, in \cref{sec:overview} we give a brief overview of our estimator.

In \cref{sec:cert-central}, we describe an algorithm for a twist on the mean estimation problem, called the \emph{certification} problem.
The main lemma analyzes an SDP whose solutions capture information about quantiles of a set of high-dimensional vectors.
It is the key tool in the design of our algorithm to estimate the mean.
This section requires no background on SoS.

Then, in \cref{sec:preliminaries} we give some formal definitions and standard theorems about SoS.
In \cref{sec:main-alg} we prove our main theorem from technical lemmas, whose proofs can be found in the appendix.

\subsection{The median of means paradigm}
\label[section]{sec:median-of-means-paradigm}
The \emph{median of means} is an approach to mean estimation for heavy-tailed distributions which combines the reduction in variance offered by averaging independent samples (thus achieving $1/\sqrt{n}$ convergence rates) with the outlier-robustness of the median (thus achieving $\sqrt{\log(1/\delta)}$ tail behavior) \cite{nemirovsky1983problem, jerrum1986random,alon1999space}.
Consider the $d=1$ case first.
Suppose $X_1,\ldots,X_n$ are i.i.d. copies of a real-valued random variable $X$ with mean $\mu \in \R$ and variance $\sigma^2$.
Let $k = \Theta(\log 1/\delta)$ be an integer, and for $i \leq k$ let $Z_i$ be the average of samples $X_{i \cdot n/k}$ to $X_{(i+1) \cdot n/k}$.\footnote{Throughout the paper we will assume that $n$ is divisible by $C \log(1/\delta)$ for an appropriate constant $C$. One may achieve this from general $n,k$ and $\delta \geq 2^{-O(n)}$ by throwing out samples to reach the nearest multiple of $C \log(1/\delta)$; the effect on the error rates is only a constant.}
Then it is an exercise to show that the median (or indeed any fixed quantile) of the $Z_i$'s satisfies
\[
\Pr \left \{ |\text{median}(Z_1,\ldots,Z_k) - \mu| > C \sigma \sqrt{\frac{\log(1/\delta)}{n}} \right \} \leq \delta
\]
for some universal constant $C$ (given the correct choice of $k$).
There are estimators achieving this $\sqrt{\log(1/\delta)}$ rate using ideas other than the median of means in the case $d=1$ \cite{catoni2012challenging, devroye2016sub}, but we focus here on median of means since it is the only approach known to prove a theorem like \cref{thm:lugosi-mendelson} in the high dimensional case.

Correctly extending this median of means idea to higher dimensions $d$ is not simple.
Suppose that $X$ is $d$-dimensional, with mean $\mu$ and covariance $\Sigma$.
Replacing $X_1,\ldots,X_n \in \R^d$ with grouped averages $Z_1,\ldots,Z_k \in \R^d$ remains possible, but the sticking point is to choose an appropriate notion of median or quantile in $d$ dimensions.

A first attempt would be to use as a median of $Z_1,\ldots,Z_k$ any point in $\R^d$ which has at most some distance $r$ to at least $ck$ of $Z_1,\ldots,Z_k$ for some $c > 1/2$.
Let us call such a point a \emph{simple $r$-median}.
It is straightforward to prove, by the same ideas as in the $d=1$ case, that
\[
  \Pr \left \{ \Norm{\mu - Z_i} > C \sqrt{\frac{\Tr \Sigma \log(1/\delta)}{n}} \text{ for at least $ck$ vectors $Z_i$} \right \} \leq \delta
\]
for some universal constant $C = C(c)$.
It follows that with probability at least $1-\delta$ the mean $\mu$ is a simple $r$-median for $r = C\sqrt{\Tr \Sigma \log(1/\delta) / n}$.
When $c > 1/2$, any two simple $r$-medians must each have distance at most $r$ to some $Z_i$, so by the triangle inequality,
\begin{align}\label[equation]{eq:simple-median}
  \Pr \left \{ \Norm{\text{simple $2r$-median} (Z_1,\ldots,Z_k) - \mu} > 2 C \sqrt{\frac{\Tr \Sigma \log(1/\delta)}{n}} \right \} \leq \delta
\end{align}
where $\text{simple $2r$-median}(Z_1,\ldots,Z_k)$ is any simple $2r$-median of $Z_1,\ldots,Z_k$.
At the cost of replacing $2r$ by $4r$, a simple $r$-median can be found easily in polynomial time (in fact in quadratic time) because if there is any simple $2r$-median of $Z_1,\ldots,Z_k$ then by triangle inequality some $Z_i$ must be a simple $4r$-median.

In prior work, Minsker shows that the geometric median of $Z_1,\ldots,Z_k$ achieves the same guarantee \cref{eq:simple-median} as the simple median (perhaps with a different universal constant $C$) \cite{minsker2015geometric}.
Geometric median is computable in nearly-linear time (that is, time $dk \cdot (\log dk)^{O(1)}$) \cite{DBLP:conf/stoc/CohenLMPS16}.

The guarantee \cref{eq:simple-median} represents the smallest confidence intervals previously known to be achievable by polynomial-time computable mean estimators under the assumption that $X$ has finite mean and covariance.
This tail bound is an intermediate between the $\sqrt{\Tr \Sigma / \delta n}$-style tail bound achieved by the empirical mean \cref{eq:bad-concentration} and the Gaussian-style guarantee of Lugosi and Mendelson from \cref{thm:lugosi-mendelson}.
It fails to match \cref{thm:lugosi-mendelson} because the $\log(1/\delta)$ term multiplies $\Tr \Sigma$ rather than $\|\Sigma\|$ -- this introduces an unnecessary dimension-dependence.
That is, if $X$ has covariance identity, then informally speaking the rate of tail decay has a dimension-dependent factor when it should be dimension-independent: it decays as $\exp(-ct^2/d)$ rather than $\exp(-ct^2)$ (where $c$ is some fixed constant).\footnote{Of course, formally we are talking about one estimator $\hat{\mu}_\delta$ for every $\delta$, so it is not correct to speak of tail decay with respect to $\delta$.}
This is not a failure of the analysis: if the approach is to draw a ball around the population mean $\mu$ which contains at least a constant fraction of $Z_1,\ldots,Z_k$ with probability $1-\delta$, the ball must have radius of order $\sqrt{\Tr \Sigma \log(1/\delta)/n}$, which grows with the dimension of $X$.

To prove \cref{thm:lugosi-mendelson}, Lugosi and Mendelson introduce a new notion of high-dimensional median, which arises from what they call a \emph{median of means tournament}.
This tournament median of $Z_1,\ldots,Z_k$ is
\begin{align}\label[equation]{eq:lm-estimator}
\arg \min_{x \in \R^d} \max_{y \in \R^d} \|x - y \| \text{ such that $\|Z_i - x\| \geq \|Z_i - y\|$ for at least $\tfrac k 2$ $Z_i$'s.}
\end{align}
Rephrased, the tournament median is the point $x \in \R^d$ minimizing the number $r$ such that for every unit $u \in \R^d$, the projection $\iprod{x,u}$ is at distance at most $r$ from a median of the projections $\{\iprod{Z_i,u}\}$.\footnote{Thanks to Jerry Li for pointing out this reinterpretation of the tournament median to me.}

In fact, Lugosi and Mendelson's arguments apply to any $x$ which \emph{$r$-central} in the following sense: for every unit $u$, there are at least $0.51k$ vectors among $Z_1,\ldots,Z_k$ such that $|\iprod{Z_i,u} - \iprod{x,u}| \leq r$.
Their proof shows that an estimator which outputs any $r$-central point will achieve the guarantee in \cref{thm:lugosi-mendelson}.
This interpretation shows that their estimator is related to a weak notion of Tukey median: a Tukey median (at least in the typical case that it has constant Tukey depth) should be between a $49$-th and $51$-st percentile in every direction $u$, while an $r$-central point has distance at most $r$ to such a percentile in every direction $u$ \cite{tukey1960survey}.
Thus our result \cref{thm:main} adds to several in the literature which demonstrate that although the Tukey median of vectors $v_1,\ldots,v_k \in \R^d$ is NP-hard to compute if $v_1,\ldots,v_k$ are chosen adversarially, under reasonable assumptions (in this case that $Z_1,\ldots,Z_k$ are i.i.d. from a distribution with bounded covariance) one may find some kind of approximate Tukey median in polynomial time \cite{bernholt2006robust,DBLP:conf/focs/DiakonikolasKK016,DBLP:conf/focs/LaiRV16}.

The heart of the proof of \cref{thm:lugosi-mendelson} shows that with probability at least $1-\delta$, the mean $\mu$ is $r$-central for $r = C(\sqrt{\Tr \Sigma / n} + \sqrt{\|\Sigma\| \log(1/\delta) / n})$.
The difficulty in \emph{computing} the tournament median -- or finding some $r$-central point -- comes from the fact that in each direction $u$ it may be different collection of $0.51k$ vectors which satisfy $|\iprod{Z_i,u} - \iprod{x,u}| \leq r$.
Thus even if an algorithm is given $Z_1,\ldots,Z_k$ \emph{and $\mu$}, to efficiently check that $\mu$ is a tournament median or is $r$-central seems naively to require brute-force search over $\exp(cd)$ directions in $\R^d$, for some fixed $c > 0$.
The heart of our algorithm is a semidefinite program which (with high probability) can efficiently \emph{certify} that $\mu$ is $r$-central: this algorithm is described in \cref{sec:cert-central}.

\subsection{Semidefinite programming and the SoS method in statistics}
One of the main tools in our algorithm is semidefinite programming, and in particular the sum of squares method.
Recall that a semidefinite program (SDP) is a convex optimization problem of the following form:
\begin{align}\label[equation]{eq:sdp}
\min_{X} \, \iprod{X,C} \text{ such that } \iprod{A_1,X} \geq 0, \ldots, \iprod{A_m,X} \geq 0 \text{ and } X \succeq 0
\end{align}
where $X$ ranges over symmetric $n \times n$ real matrices and $\iprod{M,N} = \Tr MN^\top$.
Subject to mild conditions on $C$ and $A_1,\ldots,A_m$, semidefinite programs are solvable to arbitrary accuracy in polynomial time \cite{boyd2004convex}.

Semidefinite programming as a tool for algorithm design has by now seen numerous uses across both theoretical computer science and statistics.
Familiar SDPs in statistics include the nuclear-norm minimization SDP, used for matrix sensing and matrix completion \cite{DBLP:journals/focm/CandesR09,DBLP:journals/tit/CandesT10}, the Goemans-Williamson cut SDP, variants of which are used for community detection in sparse graphs \cite{guedon2016community,montanari2016semidefinite,abbe2016exact}, SDPs for finding sparse principal components \cite{DBLP:journals/siamrev/dAspremontGJL07, DBLP:conf/isit/AminiW08, krauthgamer2015semidefinite}, SDPs used for high-dimensional change-point detection \cite{wang2018high}, SDPs used for optimal experiment design \cite{vandenberghe1998determinant}, and more.

While much work has focused on detailed analyses of a small number of canonical semidefinite programs -- the nuclear-norm SDP, the Goemans-Williamson SDP, etc. -- the SoS method offers a rich variety of semidefinite programs suited to many purposes \cite{shor1987approach,nesterov2000squared,lasserre2001global,parrilo2000structured}.
For every \emph{polynomial optimization problem with semialgebraic constraints}, SoS offers a \emph{hierarchy} of SDP relaxations.
That is, for every collection of multivariate polynomials $p,q_1,\ldots,q_m \in \R[x_1,\ldots,x_n]$ and every even $r \geq \max (\deg p, \deg q_1, \ldots, \deg q_m)$, SoS offers a relaxation of the problem
\[
\min p(x) \text{ such that } q_1(x) \geq 0, \ldots, q_m(x) \geq 0\mper
\]
As $r$ increases, the relaxations become stronger, more closely approximating the true optimum value of the optimization problem, but the complexity of the relaxations also increases.
Typically, the $r$-th relaxation is solvable in time $(nm)^{O(r)}$.
In many applications, such as when $q_1,\ldots,q_m$ include the constraints $x_i^2 - x \geq 0, x_i^2 - x \leq 0$ which imply $x \in \{0,1\}^n$, when $r = n$ the SoS SDP exactly captures the optimum of the underlying polynomial optimization problem.
However, the resulting SDP has at least $2^n$ variables, so is not generally solvable in polynomial time.
This paper focuses on SoS SDPs with $r = O(1)$ (in fact $r=8$), leading to polynomial-time algorithms.

SoS carries at least two advantages relevant to this paper over more classical approaches to semidefinite programming.
First is the flexibility which comes from the possibility of beginning with any set of polynomials $p,q_1,\ldots,q_m$; we choose polynomials which capture the idea of $r$-centrality.
Second is ease of analysis: SoS SDPs in statistical settings are amenable to an analysis strategy which converts proofs of statistical identifiability into analysis of an SDP-based algorithm by phrasing the identifiability proof as a dual solution to the SDP.
This style of analysis is feasible in our case because the SoS SDP has enough constraints that many properties of $r$-centrality carry over to the relaxed version: it is not clear whether a more elementary SDP would share this property.

\subsection{Algorithm Overview}
\label[section]{sec:overview}

Recall where we left off in \cref{sec:median-of-means-paradigm}.
Having taken samples $X_1,\ldots,X_n$ from a distribution with mean $\mu$ and covariance $\Sigma$ and averaged groups of $n/k$ of them to form vectors $Z_1,\ldots,Z_k$, the goal is to find a median of $Z_1,\ldots,Z_k$.
As we discussed, the appropriate notion of a median is any point $x \in \R^d$ which is $r$-central for $r = O(\sqrt{\Tr \Sigma / n} + \sqrt{\|\Sigma \| \log(1/\delta) / n})$, meaning that for every $1$-dimensional projection $\iprod{x,u},\iprod{Z_1,u},\ldots,\iprod{Z_k,u}$, the point $\iprod{x,u}$ has distance at most $r$ to a $0.51$-quantile of of $\{ \iprod{Z_i,u} \}$

Let us change the problem temporarily with a thought experiment: imagine being given $Z_1,\ldots,Z_k$ \emph{and} the population mean $\mu$ and being asked to verify (or in computer science jargon, \emph{certify}) that indeed $\mu$ is $r$-central.
Even for this apparently simpler task there is no obvious polynomial-time algorithm: a brute-force inspection of $\{\iprod{Z_i,u} \}$ for, say, all $u$ in an $\e$-net of the unit ball in $\R^d$ will require time $(1/\e)^d$.

Our first technical contribution is to show that with high probability over $Z_1,\ldots,Z_k$ there is a \emph{short certificate}, or \emph{witness}, to the fact that the population mean $\mu$ has distance at most $r$ to a median in every direction.
This certificate takes the form of a dual solution to a semidefinite relaxation of the following combinatorial optimization problem: given $Z_1,\ldots,Z_k,\mu$ and $r > 0$, maximize over all directions $u$ the number of $i \in [k]$ such that $\iprod{Z_i-\mu,u} \geq r$.
Solving this SDP gives an algorithm for the certification problem: we show that with probability at least $1-\delta$ the maximum value is at most $k/3$ for the choice of $r$ above.
We note that this SDP and its analysis do not rely on the SoS technology, so all of \cref{sec:cert-central} can be read without this background.

Returning to the problem of estimating $\mu$ given $Z_1,\ldots,Z_k$, the task is made simpler by the existence of the certificate that $\mu$ is $r$-central.
In particular, it gives a concrete object which our estimation algorithm can search for: we know it will suffice to find any point in the set:
\begin{align*}
  & \textsc{certifiable-centers} (Z_1,\ldots,Z_k)\\
  & = \{ (x,M) \, : \, x \in \R^d, M \in \R^{(d+k+1) \times (d+k+1)} \text{ certifies $x$ is $r$-central } \}\mcom
\end{align*}
which is nonempty because it in particular contains $(\mu,M_\mu)$, where $M_\mu$ is the aforementioned SDP dual solution.
(It is not yet obvious why $(d+k+1) \times (d+k+1)$ is the appropriate dimension for $M$; we will see this in the next section.)
Our second technical contribution is an algorithm which we call \textsc{median-sdp}, based on the SoS method, which takes $Z_1,\ldots,Z_k$ and finds $x' \in \R^d$ such that $\|x - x'\| = O(r)$ for every $x \in \textsc{certifiable-centers}$.

The algorithm is based on an SDP relaxation of the set \textsc{certifiable-centers}, this time based on the SoS method.
The relaxation is designed to accommodate the following kind of analysis: we turn the following simple argument about $r$-central points into a dual solution to the SDP (in the SoS context this object is called an \emph{SoS proof}), then use the latter to show that the SDP finds a good estimator $x$.

The argument which we must turn into an SoS proof is the following: if $x,x'$ are $r$-central then consider in particular the direction $v = (x-x')/\|x-x'\|$.
There exists some $Z_i$ such that $\iprod{x,v} \leq r + \iprod{Z_i,v}$ and $\iprod{x',-v} \leq r + \iprod{Z_i,-v}$.
Adding the inequalities gives $\iprod{x-x',v} = \|x - x'\| \leq 2r$.
When we make this argument into an SoS proof, it will imply (roughly speaking) not just when $x$ is $r$-central but also when $x$ is in our relaxation of the set \textsc{certifiable-centers}.

This strategy will rely crucially on both the existence of the certificate $\mu$ (needed to turn the above argument into an SoS proof) and the SoS strategy for designing SDPs (to accommodate the complexity of the resulting dual solution).
For more discussion, see \cref{sec:median-sdp}.

\section{Certifying Centrality}
\label[section]{sec:cert-central}

In this section we describe and analyze one of the key components of our algorithm: a semidefinite program to certify the main property of the population mean our algorithm exploits -- $(r,p)$-centrality.

\begin{definition}[Centrality]
  Let $Z_1,\ldots,Z_k \in \R^d$, $r > 0$, and $p \in [0,1]$.
  We say that $x \in \R^d$ is $(r,p)$-central (with respect to $Z_1,\ldots,Z_k$) if for every unit $u \in \R^d$ there are at most $pk$ vectors $Z_1,\ldots,Z_k$ such that $\iprod{Z_i-x,u} \geq r$.
\end{definition}

At the heart of Lugosi and Mendelson's mean estimator is the following remarkable lemma, characterizing centrality of the population mean.

\begin{lemma}[\cite{LM18}, rephrased]
\label[lemma]{fact:lm-central}
  Let $Z$ be a $d$-dimensional random vector with mean $\mu = \E Z$ and covariance $\Sigma$.
  Let $Z_1,\ldots,Z_k$ be i.i.d. copies of $Z$.
  With probability at least $1-2^{-\Omega(k)}$, the population mean $\mu$ is $(r,1/3)$-central with respect to $Z_1,\ldots,Z_k$, for $r = O(\sqrt{\Tr \Sigma / k} + \sqrt{\|\Sigma\|})$.\footnote{We write $f(n) = O(g(n))$ if there is a constant $C$ such that for all large-enough $n$ one has $f(n) \leq Cg(n)$. Similarly, we write $f = \Omega(g(n))$ if there is $c$ such that $f(n) \geq cg(n)$ for large-enough $n$. We write $f = \Theta(g(n))$ if both $f = O(g(n))$ and $f = \Omega(g(n))$.}
\end{lemma}

The main difficulty in proving \cref{fact:lm-central} (and our later algorithmic versions of it) is to simultaneously obtain the tight quantitative bound $r = O(\sqrt{\Tr \Sigma / k} + \sqrt{\|\Sigma\|})$ and the high probability $1-2^{-\Omega(k)}$.
Without both, one does not get an estimator matching \cref{thm:lugosi-mendelson}.

Suppose, as in the median of means paradigm, $Z$ is taken as the empirical average of $n/k$ i.i.d. copies $X_1,\ldots,X_{n/k}$ of another random vector $X$ having covariance $\Sigma'$.
Then $\Sigma = \tfrac kn \Sigma'$.
One may see that if $k = \Theta(\log(1/\delta))$ the mean $\mu$ is $(r,1/3)$-central for $r = O(\sqrt{\Tr \Sigma'/n} + \sqrt{\|\Sigma'\| \log(1/\delta)/n})$ with probability at least $1-\delta$.
Any two $(r,1/3)$ central points $x,y$ also have $\|x - y\| \leq 2r$ (see \cref{sec:overview}), and thus it follows that to obtain the guarantees of \cref{thm:lugosi-mendelson}, given $Z_1,\ldots,Z_k$ one only needs to output any $(r,1/3)$-central point.

\subsection{Certification and the Failure of Empirical Moments}
A natural avenue to designing an efficient algorithm matching \cref{thm:lugosi-mendelson} is to try to compute an $(r,1/3)$-central point given $Z_1,\ldots,Z_k$.
A first roadblock is that there is not an obvious efficient algorithm for the following apparently simpler problem: given $x \in \R^d$, decide whether $x$ is an $(r,1/3)$-central point -- brute-force search over $2^d$ one-dimensional projections must be avoided.
In this section we give an efficient algorithm for a slight twist of this problem, which we call the \emph{certification} problem.

\begin{problem}[Certification]
  Given $Z_1,\ldots,Z_k,x \in \R^d$ and $r > 0$ and $p \in [0,1]$, a certification algorithm may output \textsc{yes} or \textsc{do not know}.
  If the output is \textsc{yes}, then $x$ must be $(r,p)$-central with respect to $Z_1,\ldots,Z_k$.
  If the output is \textsc{do not know}, then $x$ may or may not be $(r,p)$-central.
\end{problem}

Our goal is to design a certification algorithm with parameters matching \cref{fact:lm-central}.
That is, we would like a certification algorithm which outputs \textsc{yes} with probability at least $1-2^{-\Omega(k)}$ over $Z_1,\ldots,Z_k$ when given $x = \mu$ and $r = O(\sqrt{\Tr \Sigma / k} + \sqrt{\|\Sigma\|})$ and $p$ a small constant.
This is an easier task than deciding $(r,p)$-centrality exactly, since we care only about those configurations of $Z_1,\ldots,Z_k$ which may arise as i.i.d. copies of a random vector $Z$ with covariance $\Sigma$, and even when $\mu$ is $(r,p)$-central we allow the algorithm to output \textsc{do not know}, so long as this does not happen too often.
We prove the following theorem, which we view as an algorithmic version of \cref{fact:lm-central}.

\begin{theorem}\label[theorem]{thm:sdp-central}
  There is an algorithm for the certification problem with running time $(kd)^{O(1)}$ and the guarantee that if $Z_1,\ldots,Z_k$ are i.i.d. copies of a random variable $Z$ with mean $\mu$ and covariance $\Sigma$ then the algorithm outputs \textsc{yes} with probability at least $1-2^{-\Omega(k)}$ given $p = 1/100$\footnote{The constant $1/100$ differs from the $1/3$ in \cref{fact:lm-central} only for technical convenience later in this paper.}
 and $r = O(\sqrt{\Tr \Sigma / k} + \sqrt{\|\Sigma\|})$.
\end{theorem}

Our algorithm for the certification problem will be based on semidefinite programming.
While our final algorithm to estimate $\mu$ (\cref{thm:main}) will not directly employ this certification algorithm as a subroutine, the semidefinite program we analyze for the latter is at the heart of the former.

\paragraph{On the Failure of Empirical Moments}
Before we describe our certification algorithm and prove \cref{thm:sdp-central}, we offer some intuition as to why a powerful tool such as semidefinite programming is necessary, by assessing simpler potential approaches to certification.
A natural approach would involve the maximum eigenvalue $\lambda = \|\overline{\Sigma}\|$ of the empirical covariance $\overline{\Sigma} = \frac 1k \sum_{i=1}^k (Z_i - \mu)(Z_i - \mu)^\top$.
If a unit vector $u$ has $\iprod{Z_i-\mu,u} \geq r$ for more than $k/3$ vectors $Z_i$ (thus violating $(r,1/3)$-centrality), then $\frac 1 k \sum \iprod{Z_i - \mu,u}^2 \geq r^2/3$.
Thus the maximum eigenvalue $\lambda$ (which is of course computable in polynomial time) would certify that $\mu$ is $(O(\sqrt{\lambda}),1/3)$-central.

Unfortunately, because of our weak assumptions on $Z$ -- again, we only assume the second moment $\Sigma$ exists -- the maximum eigenvalue of the empirical covariance is poorly concentrated: for instance, with probability about $2^{-k}$ some vector $Z_i$ may have norm as large as $\sqrt{\Tr \Sigma} \cdot 2^k$, resulting in $\sqrt{\lambda} \geq \sqrt{\Tr \Sigma} \cdot 2^{k/2}$.
(Indeed, even the typical value of $\sqrt{\lambda}$ could be much larger than $\sqrt{\Tr \Sigma / k} + \sqrt{\|\Sigma\|}$.)
Straightforward approaches to address this -- e.g. discarding a constant fraction of the samples $Z_1,\ldots,Z_k$ of largest norm, or replacing the second moment $\frac 1k \sum_{i=1}^k \iprod{Z_i-\mu,u}^2$ with the first moment $\frac 1k \sum_{i=1}^k |\iprod{Z_i-\mu,u}|$ -- offer some quantiative improvement over the empirical covariance, but still do not match the $\sqrt{\Tr \Sigma / k} + \sqrt{\|\Sigma\|}$ bound with probability $1-2^{-\Omega(k)}$ which we are aiming for.
Our semidefinite programming-based algorithm for certification can be viewed as a more sophisticated approach to improve the outlier-robustness of the maximum eigenvalue of the empirical covariance.

\subsection{The Centrality SDP}
We turn to our certification algorithm and the proof of \cref{thm:sdp-central}.
To start, we design a convex relaxation of the following (non-convex) optimization problem, which captures centrality:
given $Z_1,\ldots,Z_k,x$ and $r \geq 0$, find the minimum $p$ such that $x$ is $(r,p)$-central.
Or, rephrased, find the \emph{maximum} over directions $u$ of the number of $Z_i$ such that $\iprod{Z_i - x,u} \geq r$.
The latter we capture as the following \emph{quadratic program}.

\begin{fact}\label[fact]{fact:qp}
  The minimum $p$ such that $x$ is $(r,p)$-central with respect to $Z_1,\ldots,Z_k \in \R^d$ is given by the optimum of the following quadratic program in variables $b_1,\ldots,b_k$ and $u_1,\ldots,u_d$.
  \begin{align}
  & \max_{u,b} \frac 1k \sum_{i=1}^k b_i  \quad \text{such that} \label{eq:qp} \\
  & b_1,\ldots,b_k \in \{0,1\} \nonumber \\
  & \|u\|^2 \leq 1 \nonumber \\
  & b_i \iprod{Z_i -x,u} \geq b_i r \quad \text{ for } i = 1,\ldots,k\mper \nonumber
  \end{align}
\end{fact}

We relax the quadratic program \eqref{eq:qp} to a semidefinite program in standard fashion.

\begin{definition}[Centrality SDP]
  Given $Z_1,\ldots,Z_k,x \in \R^d$ and $r \geq 0$, we define a semidefinite program over $(d+k+1) \times (d+k+1)$ positive semidefinite matrices with the following block structure:
  \[
  Y(B,W,U,b,u) = \left ( \begin{array}{ccc} 1 & b^\top & u^\top\\
                                            b & B & W \\
                                            u & W^\top & U \end{array} \right )
  \]
  where $B \in \R^{k \times k},U \in \R^{d \times d},b \in \R^k, u \in \R^d$.
  As usual, the intended solutions of the SDP are rank-one matrices $(1,b,u)(1,b,u)^\top$ where $(b,u) \in \R^{d+k}$ is a solution to \eqref{eq:qp}.
  The SDP is:
  \begin{align*}
  & \max_{Y(B,W,U,b,u)} \, \frac 1k \sum_{i=1}^k b_i \quad \text{such that}\\
  & B_{ii} \leq 1 \quad \text{ for } i = 1,\ldots,k\\
  & \Tr U \leq 1\\
  & \iprod{Z_i - x, W_i} \geq r \cdot b_i \text{ for } i = 1,\ldots,k\\
  & Y(B,W,U,b,u) \succeq 0\mper
  \end{align*}
  Here $W_i$ is the $i$-th row of the $k \times d$ matrix $W$.
  It stands in for the vector $b_i \cdot u$ in \eqref{eq:qp}.\footnote{We remark that a more traditional SDP relaxation might only involve the large $(d+k) \times (d+k)$ block of $Y$, replacing $b_i$ with $B_{ii}$ in all constraints. However, the extra row and column $(1,b,u)$ will be of some technical use later in this paper; it is possible with some technical modifications to other proofs they could be removed.}
\end{definition}

The centrality SDP is a relaxation of centrality proper: there is no \emph{a priori} reason to believe that it faithfully captures the quadratic program \eqref{eq:qp}.
For instance, it could be that for most $Z_1,\ldots,Z_k$ the $r$-centrality SDP value is $1$, even though \cref{fact:lm-central} says that with high probability the value of \eqref{eq:qp} is at most $1/3$ in the median of means setting (for appropriate choice of $r$).

Remarkably, the opposite is true: at least in our median of means setting, the centrality SDP is a good approximation to the quadratic program it relaxes.\footnote{Here we do not mean approximation in the sense the word is used in \emph{approximation algorithms}, since we are studying only the behavior of the SDP for $Z_1,\ldots,Z_k$ being a collection of random vectors, and we prove only high probability guarantees, rather than probability-$1$ guarantees.}
This is captured by the following key technical lemma, from which \cref{thm:sdp-central} follows immediately (because the SDP can be solved in polynomial time \cite{boyd2004convex}).

\begin{definition}[Certifiable Centrality]
  Let $Z_1,\ldots,Z_k \in \R^d$, $r > 0$, and $p \in [0,1]$.
  We say that $x \in \R^d$ is \emph{certifiably} $(r,p)$-central (with respect to $Z_1,\ldots,Z_k$) if the value of the centrality SDP with parameters $Z_1,\ldots,Z_k,x,r$ is at most $p$.
\end{definition}

\begin{lemma}
\label[lemma]{lem:sdp-central}
  Let $Z$ be a $d$-dimensional random vector with mean $\mu = \E Z$ and covariance $\Sigma$.
  Let $Z_1,\ldots,Z_k$ be i.i.d. copies of $Z$.
  With probability at least $1-2^{-\Omega(k)}$, $\mu$ is certifiably $(O(\sqrt{\Tr \Sigma/k} + \sqrt{\|\Sigma\|}), 1/100)$-central.
\end{lemma}

Since the centrality SDP can be solved in polynomial time, \cref{lem:sdp-central} comprises an analysis of the following algorithm for the certification problem: given $Z_1,\ldots,Z_k,x$, solve the centrality SDP, and output \textsc{yes} if the optimum value is at most $1/100$ (otherwise output \textsc{do not know}).

In the rest of this section we prove \cref{lem:sdp-central}.
The proof follows a similar strategy to that used by Lugosi and Mendelson to prove \cref{fact:lm-central}.
We find it surprising that this is possible, given that Lugosi and Mendelson's argument only needs to address the quadratic program \eqref{eq:qp} (almost equivalently, it would only address rank-one solutions to the centrality SDP), while we need to argue about all relaxed solutions.

We will be able to establish, however, that the properties of \eqref{eq:qp} used by (an adaptation of) Lugosi and Mendelson's proof also hold for the centrality SDP.
In particular, we will use a bounded-differences property of the centrality SDP to establish concentration.
While bounded-differences arguments are standard, using bounded differences to show exponential concentration of the optimum value of a convex program appears to be novel.

\subsection{Proof of \cref{lem:sdp-central}}

We need to assemble a few tools for the proof of \cref{lem:sdp-central}.
The first concern the $2 \rightarrow 1$ norm of a matrix -- in particular, we will be interested in the matrix $M$ with rows $Z_1,\ldots,Z_k$.

For our purposes, the $2 \rightarrow 1$ norm of $M$ serves as a moderately outlier-robust modification of the spectral norm (a.k.a. $2 \rightarrow 2$ norm) of the empirical covariance of $Z_1,\ldots,Z_k$.
This robustness is achieved by replacing an $\ell_2$ norm with an $\ell_1$ norm.
We say ``moderately'' outlier robust because under our $2$nd moment assumption on $Z_1,\ldots,Z_k$ we will only be able to establish bounds \emph{in expectation} on the $2 \rightarrow 1$ norm of $M$, rather than high-probability bounds.

\begin{definition}\label[definition]{def:2-to-1}
  Let $A \in \R^{n \times m}$ be a matrix with rows $A_1,\ldots,A_n$.
  The $2$-to-$1$ norm of $A$ is defined as
  \[
  \|A\|_{2 \rightarrow 1} = \max_{\|u\|=1} \|Au\|_1 = \max_{\|u\|=1,\sigma \in \{\pm 1\}^n} \sum_{i \leq n} \sigma_i \iprod{A_i,u}\mper
  \]
\end{definition}

Computing the $2 \rightarrow 1$-norm of a matrix $A$ exactly is computationally intractable \cite{bhattiprolu2018inapproximability}.
Nonetheless, we will profitably use a convex program -- again, an SDP -- whose optimal values can be related to the $2 \rightarrow 1$ norm.
Eventually we will relate the centrality SDP to the following slightly different SDP.
It is one of a well-studied family of SDPs for $p \rightarrow q$-norm problems, the most famous of which is the $\infty \rightarrow 1$-norm SDP appearing in Grothendieck's inequality and used to approximate the cut norm of a matrix \cite{DBLP:journals/siamcomp/AlonN06}.

\begin{definition}
  \label[definition]{def:2-to-1-sdp}
  For $n,m \in \N$ let $\cS_{n,m}^{2 \rightarrow 1}$ be the following subset of $\R^{(n+m) \times (n+m)}$, treated as the set of block matrices
  \[
  X(S,R,U) = \left ( \begin{array}{cc} S & R \\ R^\top & U \end{array} \right )\mper
  \]
  with $S \in \R^{n \times n}$ and $U \in \R^{m \times m}$.
  \[
  \cS_{n,m}^{2 \rightarrow 1} = \{ X(S,R,U) \, : \,   S_{ii} = 1 \text{ for } i = 1,\ldots,n, \, \,  \Tr U \leq 1, \text{ and } X \succeq 0 \}\mper
  \]
  Here we think of $S$ as a relaxation of rank-one matrices $\sigma \sigma^\top$, where $\sigma \in \{ \pm 1\}$ is as in \cref{def:2-to-1}, and $U$ as a relaxation of $uu^\top$ where $u$ is a unit vector as in \cref{def:2-to-1}.
\end{definition}

The following theorem is due to Nesterov.
It will allow us to control the optimum value of an SDP relaxation of the $2 \rightarrow 1$ norm in terms of the $2 \rightarrow 1$ norm itself.
It follows fairly easily from the observation that $\|A\|_{2 \rightarrow 1}^2 = \max_{\sigma \in \{\pm 1\}^n} \sigma^\top A^\top A \sigma$ and the fact (also due to Nesterov) that semidefinite programming yields a $\tfrac 2 \pi$-approximation algorithm  for the maximization of a positive semidefinite quadratic form over $\{ \pm 1\}^n$ (see e.g. \cite{DBLP:books/daglib/0030297}, section 6.3 for a simple proof).

\begin{theorem}[\cite{nesterov1998semidefinite}]\label[theorem]{thm:sos-2-to-1}
  There is a constant $K_{2 \rightarrow 1} = \sqrt{\pi/2} < 2$ such that for every $n \times m$ matrix $A$, one has the following inequality:
  \[
  \max_{X(S,R,U) \in \cS_{n,m}^{2 \rightarrow 1}} \iprod{R,A} \leq K_{2 \rightarrow 1} \|A\|_{2 \rightarrow 1}\mper
  \]
\end{theorem}

The following lemma affords control over $\E \|M\|_{2 \rightarrow 1}$, where $M$ has rows $Z_1,\ldots,Z_k$.
The proof uses standard tools from empirical process theory; a similar argument appears in \cite{LM18}.
We provide the proof in \cref{sec:omitted-centrality}.

\begin{lemma}\label[lemma]{lem:2-to-1}
  Let $Z$ be an $\R^d$-valued random variable with mean $\E Z = 0$ and covariance $\E ZZ^\top = \Sigma$.
  Let $Z_1,\ldots,Z_k$ be iid copies of $Z$, and let $M \in \R^{k \times d}$ be the matrix whose rows are $Z_1,\ldots,Z_k$.
  Then
  \[
  \E \|M\|_{2 \rightarrow 1} \leq 2\sqrt{k \Tr \Sigma} + k \sqrt{\|\Sigma\|}
  \]
  where $\|\Sigma\|$ denotes the operator norm, or maximum eigenvalue, of $\Sigma$.
\end{lemma}

Finally, the last lemma on the way to \cref{lem:sdp-central} shows that the centrality SDP satisfies a bounded differences property: this is crucial to establishing the high-probability bound in \cref{lem:sdp-central}.
The proof is \cref{sec:omitted-centrality}.

\begin{lemma}
\label[lemma]{lem:bdd-differences}
  Let $r \geq 0$ and $x \in \R^d$.
  Let $Z_1,\ldots,Z_k \in \R^d$, $i \in [k]$ and $Z_i' \in \R^d$.
  Let $SDP(Z_1,\ldots,Z_k,x,r)$ be the optimum value of the centrality SDP with parameters $Z_1,\ldots,Z_k,x,r$.
  Then
  \[
  |SDP(Z_1,\ldots,Z_k,x,r) - SDP(Z_1,\ldots,Z_{i-1},Z_i',Z_{i+1},\ldots,Z_k,x,r)| \leq \frac 1k\mper
  \]
\end{lemma}

Now we are ready to prove \cref{lem:sdp-central}.

\begin{proof}[Proof of \cref{lem:sdp-central}]
  The proof has an expectation step and a concentration step.
  Let $SDP(Z_1,\ldots,Z_k,\mu,r)$ be the optimum value of the centrality SDP.
  Since $Z_1,\ldots,Z_k$ are independent, by the bounded differences inequality together with \cref{lem:bdd-differences},
  \[
  \Pr \left ( SDP(Z_1,\ldots,Z_k,\mu,r) - \E SDP(Z_1,\ldots,Z_k,\mu,r) > 1/200 \right ) < 2^{-\Omega(k)}\mper
  \]
  Thus, it will suffice to show that $\E SDP(Z_1,\ldots,Z_k, \mu,r) \leq 1/200$ for some $r = O(\sqrt{\Tr \Sigma / k} + \sqrt{\|\Sigma\|})$.

  By definition of the centrality SDP, using the constraints $\iprod{Z_i-x, W_i} \geq r \cdot b_{i}$, we have
  \[
  \E \max_{B,W,U,b,u} \frac 1k \sum_{i \leq k} b_i \leq \frac 1 {kr} \E \max_{B,W,U,b,u} \sum_{i\leq k } \iprod{W_i, Z_i-\mu}\mper
  \]
    Let $\cS'$ be the set $\cS_{k,d}^{2 \rightarrow 1}$ with the modified constraint $S_{ii} \leq 1$ rather than $S_{ii} = 1$.
    Then we have $\cS' \supseteq \{Y(B,W,U)\}$ where the latter is the set of feasible solutions to the centrality SDP (restricted to the large $(d+k) \times (d+k)$ block), and hence
    \[
      \frac 1 {kr} \E \max_{B,W,U,b,u} \sum_{i \leq k} \iprod{W_i, Z_i-\mu} \leq \frac 1 {kr} \E \max_{X(S,R,U) \in \cS'} \sum_{i\leq k} \iprod{R_i, Z_i-\mu}
    \]
    where $R$ has rows $R_1,\ldots,R_k$, since the left-hand side maximizes over a larger set of PSD matrices.

    We would like to replace $\cS'$ with $\cS_{k,d}^{2 \rightarrow 1}$.
    For this we need to argue that the constraints $S_{ii} = 1$ are satisfied by the optimal $X(S,R,U)$.
    First of all, note that the maximum on the right-hand side is obtained at $X(S,R,U)$ where $\iprod{R_i,Z_i-\mu} \geq 0$, otherwise we may replace $X$ with $\tfrac 12 X + \tfrac 12 (-E_{ii}) X (-E_{ii})$ and remain inside $\cS'$ while only increasing $\iprod{R_i,Z_i-\mu}$ -- here $E_{ii}$ is the matrix with exactly one nonzero entry, at the $(i,i)$-th position, with value $1$.

    Hence also the maximum is obtained at $X(S,R,U)$ with $S_{ii} = 1$, otherwise we may rescale the $i$-th row and column by $1/\sqrt{S_{ii}}$ and remain in $\cS'$ while only increasing $\iprod{R_i,Z_i-\mu}$ (here we used that $\iprod{R_i,Z_i-\mu} \geq 0$, so $\iprod{R_i,Z_i-\mu}/\sqrt{S_{ii}} \geq \iprod{R_i,Z_i-\mu}$).
    Ultimately, we can conclude that
    \[
    \frac 1 {kr} \E \max_{X(S,R,U) \in \cS'} \sum_{i \leq k} \iprod{R_i,Z_i-\mu} = \frac 1 {kr} \E \max_{X(S,R,U) \in \cS_{k,d}^{2 \rightarrow 1}} \sum_{i \leq k} \iprod{R_i, Z_i-\mu}\mper
    \]
    The right-hand side is exactly the $2 \rightarrow 1$-norm SDP relaxation from \cref{def:2-to-1-sdp}.
    So if $M$ is the matrix with rows $Z_i - \mu$, we get
    \begin{align*}
    \E SDP(Z_1,\ldots,Z_k,\mu,r) & \leq \frac {K_{2 \rightarrow 1}} {kr} \cdot \E \|M\|_{2 \rightarrow 1}\\
    & \leq \frac {K_{2 \rightarrow 1}} {r} \cdot \Paren{2\sqrt{\Tr \Sigma/k} + \sqrt{\|\Sigma\|}}\mcom
    \end{align*}
    where we have used \cref{thm:sos-2-to-1} and $K_{2 \rightarrow 1}$ is the constant from that theorem.
    By choosing $r = 1000 (\sqrt{\Tr \Sigma / k} + \sqrt{\|\Sigma\|})$ the lemma follows.
\end{proof}

\section{SoS Preliminaries}
\label[section]{sec:preliminaries}

Now that we have established certifiable centrality of the mean, we can turn back to our main goal: design an algorithm to estimate the mean $\mu$ in order to prove \cref{thm:main}.
While in \cref{sec:cert-central} we employed a traditional style of semidefinite program (arising as a relaxation of a quadratic program), to prove \cref{thm:main} we will need a larger semidefinite program (i.e. having more variables and constraints).
The sum of squares method offers a principled way to exploit the addition of extra variables and constraints to semidefinite programs.

Treating SoS-style semidefinite programs with the traditional language and notation of semidefinite programming is often cumbersome.
Recent work in theoretical computer science has pioneered an alternative point of view, involving \emph{pseudoexpectations}, which correspond to SDP primal solutions, and \emph{SoS proofs}, which correspond to SDP dual solutions.
Analyzing a complex semidefinite program can often be reduced to the construction of an appropriate dual solution.
The pseudoexpectation/SoS proof point of view is designed to make this construction possible in a modular fashion, building a complicated dual solutions out of many simpler ones.

In this section we get set up to use the SoS approach for our main algorithm.
We review the preliminaries we need and refer the reader to other resources for a full exposition -- see e.g. \cite{sos-notes-general}.

\begin{definition}[SoS Polynomials]
  Let $x = x_1,\ldots x_n$ be some indeterminates, and let $p \in \R[x]$.
  We say that $p$ is SoS if it is expressible as $p = \sum_{i=1}^m q_i(x)^2$ for some other polynomials $q_i$.
  We write $p \succeq 0$, and if $p - q \succeq 0$ we write $p \succeq q$.
\end{definition}

\begin{definition}[SoS Proof]
  Let $\cA = \{p_1(x) \geq 0, \ldots, p_m(x) \geq 0 \}$ be a set of polynomial inequalities.
  We sometimes include polynomial equations $p_i(x) = 0$, by which we mean that $\cA$ contains both $p_i(x) \geq 0$ and $-p_i(x) \geq 0$.
  We say that $\cA$ SoS-proves that $q(x) \geq 0$ if there are SoS polynomials $q_S(x)$ for every $S \subseteq [m]$ such that
  \[
  q(x) = \sum_{S \subseteq [m]} q_S(x) \prod_{i \in S} p_i(x)\mper
  \]
  The polynomials $q_S(x)$ form an \emph{SoS proof} that $q(x) \geq 0$ for every $x$ such that $p_i(x) \geq 0$.
  If $\deg q_S(x) \cdot \prod_{i \in S} p_i(x) \leq d$ for every $S$, then we say that the proof has \emph{degree $d$}, and write
  \[
  \cA \proves{d} q(x) \geq 0\mper
  \]
  SoS proofs obey many natural inference rules, which we will freely use in this paper -- see e.g. \cite{sos-notes-general}.
\end{definition}

Critically, the set of SoS proofs of $q(x) \geq 0$ using axioms $\cA$ form a convex set (in fact, a semidefinite program).
Their convex duals are called \emph{pseudodistributions} or \emph{pseudoexpectations} (we use the terms interchangeably).

\begin{definition}[Pseudoexpectation]
  A degree-$d$ \emph{pseudoexpectation} in variables $x = x_1,\ldots,x_n$ is a linear operator $\pE \, : \, \R[x]_{\leq d} \rightarrow \R$, where $\R[x]_{\leq d}$ are the polynomials in $x$ with real coefficients and degree at most $d$.
  A pseudoexpectation is:
  \begin{enumerate}
  \item Normalized: $\pE 1 = 1$, where $1 \in \R[x]_{\leq d}$ on the left side is the constant polynomial.
  \item Nonnegative: $\pE p(x)^2 \geq 0$ for every $p$ of degree at most $d/2$.
  \end{enumerate}
  \end{definition}

\begin{definition}[Satisfying constraints]
A pseudoexpectation of degree $d$ \emph{satisfies} a polynomial equation $p(x) = 0$ if for every $q(x)$ such that $p(x)q(x)$ has degree at most $d$ it holds that $\pE p(x)q(x) = 0$.
The pseudodistribution satisfies an inequality $p(x) \geq 0$ if for every $q(x)^2$ such that $\deg q(x)^2 p(x) \leq d$ it holds that  $\pE p(x) q(x)^2 \geq 0$.
\end{definition}

\begin{example}
  To demystify pseudoexpectations slightly, consider the classic semidefinite relaxation of the set $\{\pm 1\}^n$ to the set $\{X \in \R^{n \times n} \, : \, X \succeq 0, X_{ii} = 1\}$.
  (This is exactly the set of PSD matrices employed in the SDP-based \textsc{max-cut} algorithm of Goemans and Williamson \cite{DBLP:journals/jacm/GoemansW95}.)

  Each such $X$ defines a degree $2$ pseudoexpectation, by setting $\pE x_i x_j = X_{ij}$ for $1 \leq i \leq n$, $\pE x_i = 0$, and finally $\pE 1 = 1$.
  Since $X \succeq 0$, it also follows that for every polynomial $p \in \R[x_1,\ldots,x_n]_{\leq 2}$, one has $\pE p(x)^2 = p_1^\top X p_1 + \hat{p}(\emptyset)^2 \geq 0$, where $p_1$ is the vector of coefficients of the homogeneous linear part of $p$ and $\hat{p}(\emptyset)$ is the constant term in $p$.
  Last, since $\pE x_i^2 = X_{ii} = 1$, the pseudoexpectation satisfies $x_i^2 - 1 = 0$ for each $i$; these equations exactly characterize $\{\pm 1\}^n$ as a variety in $\R^n$.\footnote{In this case, $\pE$ is defined by a few more parameters than $X$ -- namely the values $\pE x_i$, which we set to zero.
  For most algorithms involving degree-$2$ pseudoexpectations the main focus is on the $n^2$ variables $\pE x_i x_j$, so this is not too surprising.
  However, as we will see in the algorithm in \cref{sec:median-sdp}, pseudoexpectations of degree higher than $2$ can contain useful information about polynomials of various degrees.}
\end{example}

As in this simple example, it is always possible to write an explicit semidefinite program whose solutions are pseudoexpectations satisfying some chosen set of polynomial inequalities.
However, as the degrees and complexity of the of polynomials grow, these SDPs become notationally unwieldy.
In this regard, the pseudoexpectation approach carries significant advantages.

The most elementary fact relating pseudodistributions and SoS proofs is the following:

\begin{fact}
  Suppose $\cA \proves{d} p(x) \geq 0$.
  Then any degree-$d$ pseudodistribution $\pE$ which satisfies $\cA$ also has $\pE p(x) \geq 0$.
\end{fact}

We will make use of the following theorem, which can be proved via semidefinite programming.

\begin{theorem}[Adapted from \cite{sos-notes-general}]\label[theorem]{thm:sos-alg}
For every $d \in \N$ there exists an $(mn)^{O(d)}$-time algorithm which given a set of $m$ $n$-variate polynomial inequalities $\cA$ which:
\begin{itemize}
  \item has coefficients with bit complexity at most $(mn)^{O(d)}$
  \item contains a constraint of the form $\|x\|^2 \leq M$ for a positive constant $M$, and
  \item is satisfied by some $x \in \R^n$
\end{itemize}
finds a degree $d$ pseudodistribution which satisfies $\cA$ up to an additive error of $2^{-(mn)^d}$ in each inequality.
\end{theorem}

In general the additive $2^{-(mn)^d}$ errors will not bother us, because the magnitudes of coefficients in the SoS proofs we construct will be bounded by $\poly(n,m)$.
See \cite{sos-notes-general,raghavendra2017bit} for more discussion of such numerical considerations.

%
%

We will use the following simple fact about pseudodistributions.

\begin{fact}\label[fact]{fact:2-norm}
  Let $\pE$ be a pseudodistribution of degree $2$ in variables $x_1,\ldots,x_n$ and let $\mu \in \R^n$.
  Then $\| \pE x - \mu \|^2 \leq \pE \|x-\mu\|^2$.
\end{fact}
 \begin{proof}
  Follows from $\pE (x_i - \mu_i)^2 \geq (\pE x_i - \mu_i)^2$ for every $i \leq n$, which follows from the more general fact $\pE p(x)^2 \geq (\pE p(x) )^2$ for every degree $1$ polynomial $p$.
  The latter follows by $\pE (p(x) - \pE p(x))^2 \geq 0$.
\end{proof}

%
%
%
%

\section{Main Algorithm and Analysis}
\label[section]{sec:main-alg}

Our main lemma for this section gives an algorithm which recovers a central point given vectors $Z_1,\ldots,Z_k$, provided that a \emph{certifiably} central point exists (and some minor additional regularity conditions on $Z_1,\ldots,Z_k$ are met).

\begin{lemma}\label[lemma]{lem:sos-median}
  For every $d,k \in \N$ and $C,r > 0$ there is an algorithm \textsc{median-sdp} which runs in time $(dk \log C)^{O(1)}$ and has the following guarantees.
  Let $Z_1,\ldots,Z_k \in \R^d$.
  Suppose that $\mu \in \R^d$ is certifiably $(r,1/100)$-central with respect to $Z_1,\ldots,Z_k$.
  And, suppose that at most $k/100$ of the vectors $Z_1,\ldots,Z_k$ have $\|Z_i - \mu\| > C r$.
  Then given $Z_1,\ldots,Z_k$, \textsc{median-sdp} returns a point $\hat{\mu}$ with $\|\mu - \hat{\mu}\| = O(r)$.
\end{lemma}

Together \cref{lem:sdp-central,lem:sos-median} suffice to prove \cref{thm:main}, with the small modification that the algorithm is given access to $r,C$ in addition to the samples $X_1,\ldots,X_n$.
We discuss in \cref{sec:avoid-assumptions} how to use standard ideas to avoid this dependence.

\begin{proof}[Proof of \cref{thm:main}]
  Let $k = c \log(1/\delta)$ for a big-enough constant $c$.
  Given samples $X_1,\ldots,X_n$, for $i \leq k$ let $Z_i$ be the average of samples $X_{i \cdot (n/k)},\ldots,X_{(i+1) \cdot n/k - 1}$ (throwing out samples as necessary so that $n$ is divisible by $k$).
  Then $Z_1,\ldots,Z_k$ are i.i.d. copies of a random variable $Z$ with $\E Z = \mu$ and $\E (Z- \mu)(Z-\mu)^\top = \tfrac kn \Sigma$.
  By \cref{lem:sdp-central}, $\mu$ is certifiably $(r,1/100)$-central with respect to $Z_1,\ldots,Z_k$ for $r = O(\sqrt{\Tr \Sigma / n} + \sqrt{\|\Sigma\| k / n})$ with probability at least $1-\exp(-\Omega(k))$.
  We can choose $c$ so that this probability is at least $1-\delta$ and $\sqrt{\|\Sigma\| k / n} = O(\sqrt{\|\Sigma\| \log(1/\delta) / n})$.

  Furthermore, by Chebyshev's inequality and a binomial tail bound, with probability at least $1 - \exp(-\Omega(k))$ we have that $\|Z_i - \mu\| \leq O(\sqrt{\Tr k \Sigma / n}) \leq O(kr)$ for all but $k/100$ vectors $Z_i$.
  Hence, except with probability $2^{-\Omega(k)}$, calling \textsc{median-sdp} with $C = O(k)$ yields a vector $x$ with $\|\mu - x\| \leq O(r)$.
\end{proof}

In the remainder of this section we prove \cref{lem:sos-median} from technical lemmas which are proved in the appendix.
We will make use of the SoS method, which will require some setup and technical arguments, so we describe the main idea first.
Given $Z_1,\ldots,Z_k$, we will define a system of polynomial equations $\cA$ whose feasible solutions are the certifiably $(r,1/10)$-central points.
(For technical convenience actually $\cA$ has feasible solutions which are the certifiably $(r,1/10)$-central points satisfying an additional mild regularity condition, as we discuss below.)
Our main algorithm will find a pseudodistribution which satisfies $\cA$ and extract from it an estimator $\hat{\mu} \in \R^d$.

To argue about $\|\hat{\mu} - \mu\|$, we will construct SoS proofs (using $\cA$ as axioms) of several inequalilties concerning certifiable $(r,1/10)$-central points.
Together these inequalities will capture the fact that any two $(r,1/10)$-central points $x,y$ have $\|x-y\| \leq 2r$; we will use the SoS proofs of these inequalities as duals to the set of pseudodistributions satisfying $\cA$, ultimately showing that $\|\hat{\mu} - \mu\| = O(r)$.

Before we can construct $\cA$, we need to observe a consequence of SDP duality -- certifiable centrality of $\mu$ implies the existence of a witness to its centrality.
(Here it may help to recall the set \textsc{certifiable-centers} from \cref{sec:intro}.)
Our construction of $\cA$ will exploit these witnesses.

\begin{lemma}
\label[lemma]{lem:witness}
  Let $Z_1,\ldots,Z_k,x \in \R^d$ and suppose $x$ is certifiably $(r,p)$-central with respect to $Z_1,\ldots,Z_k$.
  Then there are nonnegative numbers $\alpha_1,\ldots,\alpha_k, \beta_1,\ldots,\beta_k, \gamma$ and a degree-$2$ SoS polynomial $\sigma \in \R[b_1,\ldots,b_k,v_1,\ldots,v_d]_{\leq 2}$ such that the following polynomial identity holds in variables $b_1,\ldots,b_k,v_1,\ldots,v_d$.
  \begin{align}\label[equation]{eq:median-certificate}
  pk - \sum_{i=1}^k  b_i & = \sum_{i=1}^k \alpha_i b_i (\iprod{Z_i - x,v} - r)
    + \sum_{i=1}^k \beta_i (1-b_i^2) \\
    & + \gamma (1-\|v\|^2) \nonumber
    + \sigma (b,v)\mper
  \end{align}
\end{lemma}

The proof is a direct application of SDP duality -- see e.g. \cite{boyd2004convex}.
(The polynomial identity is obtained by evaluating the quadratic form of an optimal dual solution to the centrality SDP at the vector of indeterminates $(1,b,u)$.)
The numbers $\alpha,\beta,\gamma$ and SoS polynomial $\sigma$ are an SoS proof that $x$ is $(r,p)$-central: they witness
\[
\bigcup_{i \leq k} \{b_i^2 \leq 1, \|v\|^2 \leq 1, b_i \iprod{Z_i-x,v} - b_i r \geq 0 \} \proves{2} \sum_{i=1}^k b_i \leq pk\mper
\]
Indeed one may check that if $v$ is any unit vector and $b$ is the $0/1$ indicator for those $i \in [k]$ such that $\iprod{Z_i-x,v} \geq r$, then the right-hand side of \cref{eq:median-certificate} is nonnegative when evaluated at $b,v$.
Hence the left-hand side must be as well, which means that $\sum_{i \in k} b_i \leq pk$.

The last step before constructing the polynomial system $\cA$ is to observe a consequence of the regularity condition from \cref{lem:sos-median} that $\|Z_i -\mu\| \leq C r$ for at least $99k/100$ $Z_i$'s.
Namely, it affords some control over the magnitudes of the numbers $\alpha_1,\ldots,\alpha_k,\gamma$ from \cref{lem:witness}, ensuring that the witness $\alpha_1,\ldots,\alpha_k,\beta_1,\ldots,\beta_k,\gamma$ has a certain well-conditioned-ness property.
We will capture the well-conditioned-ness property in $\cA$ and make use of it in our SoS proofs.
The proof of the following lemma involves elementary manipulations on equations like \cref{eq:median-certificate}; we defer it to the appendix.

\begin{lemma}
\label[lemma]{lem:nice-witness}
  Let $Z_1,\ldots,Z_k,x \in \R^d$ and suppose $x$ is $(r,p)$-central with respect to $Z_1,\ldots,Z_k$.
  Suppose also that $\|Z_i - x\| \leq Cr$ for all but $qk$ vectors $Z_i$, where $C \geq 1$.
  Then there are nonnegative numbers $\alpha_1,\ldots,\alpha_k, \beta_1,\ldots,\beta_k, \gamma$ and a degree-$2$ SoS polynomial $\sigma \in \R[b_1,\ldots,b_k,v_1,\ldots,v_d]_{\leq 2}$ such that the following polynomial identity holds in variables $b_1,\ldots,b_k,v_1,\ldots,v_d$.
  \begin{align}
  (p+q+1/20) k - \sum_{i=1}^k  b_i & = \sum_{i=1}^k \alpha_i b_i (\iprod{Z_i - x,v} - r)
    + \sum_{i=1}^k \beta_i (1-b_i^2) \\
    & + \gamma (1-\|v\|^2) \nonumber
    + \sigma (b,v)\mper
  \end{align}
  Furthermore, $\gamma$ is in the finite set $\{0,1/100,2/100,\ldots,k \}$, and $\alpha_1,\ldots,\alpha_k$ are in the set $\{0\} \cup [1/100Cr,4k/r]$.
\end{lemma}

Now we are able to construct our main polynomial system $\cA$, whose solutions correspond to $x,\alpha,\beta,\gamma,\sigma$ such that $\alpha,\beta,\gamma,\sigma$ form a witness that $x$ is a certifiably $(r,1/10)$-central.
For technical convenience, we take $\gamma$ to be a \emph{parameter} of this system rather than one of its indeterminates.
Part of our algorithm will involve a brute-force search for a good choice of $\gamma$ -- by \cref{lem:nice-witness} there will only be $O(k)$ possibilities to search over.

\begin{definition}[The polynomial system $\cA(Z_1,\ldots,Z_k,r,C,c,\gamma)$]
  For vectors $Z_1,\ldots,Z_k \in \R^d$, $r > 0$, and $c,C > 0$ we define a system of equations in the following variables:
  \begin{align*}
    & \alpha_1,\ldots,\alpha_k, \beta_1,\ldots,\beta_k, \sigma_{ij} \text{ for } i,j \in [d+k+1], \\
    & x_1,\ldots,x_d, \text{ and } a_{i,t} \text{ for } i \in [k] \text{ and } t \in [\log C/c + 1]\mper
  \end{align*}
  Let $\cA_{\text{sos}}$ be the set of linear equations among $\alpha_1,\ldots,\alpha_k,\beta_1,\ldots,\beta_k, \sigma_{ij}, x$ which ensure that the polynomial identity
\begin{align*}
    \frac k {10} - \sum_{i=1}^k b_i
    & = \sum_{i \in S} \alpha_i b_i (\iprod{Z_i-\mu,v} - r)
    + \sum_{i=1}^k \beta_i (1 - b_i^2) \\
    & + \gamma(1-\|v\|^2)
    + \sum_{i \in [d+k+1]} \iprod{\sigma_i, (1,b,v)}^2
    \end{align*}
    holds in variables $b_1,\ldots,b_k,v_1,\ldots,v_d$, where $\sigma_i$ is the vector with $j$-th entry $\sigma_{ij}$ and $(1,b,v)$ is the $(d+k+1)$-dimensional concatenation $1,b_1,\ldots,b_k,v_1,\ldots,v_d$.
    We often abuse notation and write $\sigma(b,v)$ for the expression $\sum_{i \in [d+k+1]} \iprod{\sigma_i, (1,b,v)}^2$.
    Let $\cA_{\text{nonneg}}$ be the inequalities
    \begin{align*}
      \alpha_i \geq 0 \text{ for } i \in [k] \text{ and } \beta_i \geq 0 \text{ for } i \in [k]
    \end{align*}
    Let $\cA_{\text{a}}$ be the equations and inequalities
    \begin{align*}
      & a_{i,t}^2 = a_{i,t} \text{ for } t \in [\log C/c+1] \\
      & a_{i,t} \cdot 2^{t-1} \cdot c \leq a_{i,t} \cdot \alpha_i \text{ for } t \in [1,\log C/c+1] \\
      & a_{i,t} \cdot \alpha_i \leq a_{i,t} \cdot 2^{t} \cdot c \text{ for } t \in [1,\log C/c+1] \\
      & a_{i,0} \cdot \alpha_i = 0 \\
      & \sum_{t \leq \log C/c+1} a_{i,t} = 1 \text{ for all $i \leq k$} \\
      & a_{i,t} a_{i,t'} = 0 \text{ for all $i \leq k$ and $t \neq t'$.}
    \end{align*}
    The inequalities $\cA_a$ ensure that $a_{i,t} \in \{0,1\}$ and $a_{i,t} = 1$ if and only if $\alpha \in [2^{t-1} c, 2^{t}c]$ (or $\alpha_i = 0$ in the case of $a_{i,0}$).
    We will use the variables $a_{i,t}$ to approximate some functions of $\alpha_i$ which are not polynomials.
    For instance, if $\alpha,a$ satisfy $\cA_a$ and $\alpha_i > 0$ then $\sum_{1 \leq t \leq \log C/c+1} a_{i,t} / (c \cdot 2^{t}) \in [1/2\alpha_i, 1/\alpha_i]$.

    Finally, let $\cA = \cA_{\text{sos}} \cup \cA_{\text{nonneg}} \cup \cA_{\text{a}}$.
\end{definition}

Now we can describe the algorithm \textsc{median-sdp} and its main analysis.

\begin{algorithm}{\textsc{median-sdp}}
\label[algorithm]{alg:med-sdp}
  \textbf{Given:} $Z_1,\ldots,Z_k \in \R^d, r,C > 0$
  \begin{enumerate}
  \item For each $\gamma \in \{0,1/100,2/100,\ldots,k\}$, try to find a degree $8$ pseudodistribution satisfying $\cA(Z_1,\ldots,Z_k,r,1/100Cr, 4k/r,\gamma)$.
  If none exists for any $\gamma$, output \textsc{reject}.
  Otherwise, let $\pE$ be the pseudodistribution obtained for any $\gamma$ for which one exists.
  \item Output $\pE x$.
  \end{enumerate}
\end{algorithm}

\begin{lemma}[Main lemma for \textsc{median-sdp}]\label[lemma]{lem:sos-median-proof}
  Let $Z_1,\ldots,Z_k \in \R^d$.
  Let $\mu$ be certifiably $(r,1/10)$-central.
  Then for every $c,C,\gamma$, any degree-$8$ pseudodistribution $\pE$ satisfying $\cA$ has $\pE \|x - \mu\|^2 = O(r^2)$.
\end{lemma}

We will prove \cref{lem:sos-median-proof} in \cref{sec:median-sdp}.
We wrap up this section by proving \cref{lem:sos-median} from \cref{lem:sos-median-proof,lem:witness,lem:nice-witness}.

\begin{proof}[Proof of \cref{lem:sos-median}]
  Since at most $k/100$ of of $Z_1,\ldots,Z_k$ have $\|Z_i - \mu\| > Cr$, and because $\mu$ is $(r,1/100)$-certifiable, together \cref{lem:witness,lem:nice-witness} show that there exist nonnegative $\alpha_1,\ldots,\alpha_k, \beta_1,\ldots,\beta_k, \gamma$ and a degree-$2$ SoS polynomial $\sigma$ such that
  \begin{align*}
  0.07k - \sum_{i \leq k} b_i & =  \sum_{i=1}^k \alpha_i b_i (\iprod{Z_i - x,v} - r)
    + \sum_{i=1}^k \beta_i (1-b_i^2) \\
    & + \gamma (1-\|v\|^2) \nonumber
    + \sigma (b,v)\mper
  \end{align*}
  holds as a polynomial identity in $b,v$.
  Furthermore, $\alpha_i \in \{ 0\} \cup [1/100Cr, 4k/r]$ and $\gamma \in \{1/100,2/100,\ldots,k\}$.
  So, $\cA(Z_1,\ldots,Z_k,r,1/100Cr, 4k/r,\gamma)$ is feasible.
  Thus, \textsc{median-sdp} with parameters $r,C$ eventually finds a pseudodistribution $\pE$ satisfying $\cA$ for some $\gamma'$.
  So by \cref{lem:sos-median-proof} we have $\pE \|x-\mu\|^2 \leq O(r^2)$.
  Then the main conclusion of \cref{lem:sos-median} follows by \cref{fact:2-norm}.

  The running time bound follows by observation that $\cA$ has $(dk \log C)^{O(1)}$ variables and inequalities with this choice of parameters, then application of \cref{thm:sos-alg}.
\end{proof}

\section{Conclusion}
We have described the first polynomial-time algorithm capable of estimating the mean of a distribution with confidence intervals asymptotically matching those of the empirical mean in the Gaussian setting, under only the assumption that the distribution has finite mean and covariance.
Previous estimators with matching rates under such weak assumptions required exponential computation time.
Our algorithm uses semidefinite programming, and in particular the SoS method.
The SDP we employ is sufficiently powerful that Lugosi and Mendelson's analysis of their tournament-based estimator can be transformed to an analysis of the SoS SDP.

Our algorithm runs in polynomial time, but it is not close to practical for any substantially high-dimensional data set.
Work building on the present paper has already reduced the running time to $O(n^{3.5} + n^2 d) \cdot (\log nd)^{O(1)}$ \cite{cherapanamjeri2019fast}.
It remains an interesting direction for future study whether there is a \emph{practical} algorithm whose empirical performance improves on that of fast, practical algorithms (like geometric median) which achieve a $\sqrt{\Tr \Sigma \log(1/\delta) / n}$-style confidence interval.

\section*{Acknowledgements}
Thanks to Siva Balakrishnan and Stas Minsker for bringing the mean estimation problem to my attention.
I am most grateful to Peter Bartlett, Tarun Kathuria, Pravesh Kothari, Jerry Li, Gabor Lugosi, Prasad Raghavendra, and Jacob Steinhardt for helpful conversations as this manuscript was being prepared, and to anonymous reviewers for many suggestions in improving its presentation and correcting errors.
An earlier version of this manuscript contained a serious technical error: I am greatly indebted to Yeshwanth Cherapanamjeri for pointing it out to me (and explaining it to me several times), and to Prasad Raghavendra for several suggestions in correcting it.
Finally, thanks to the editors and anonymous reviewers of the Annals of Statistics whose suggestions substantially improved this manuscript.


\bibliographystyle{amsalpha}
\bibliography{bib/mathreview,bib/dblp,bib/custom,bib/scholar}

\appendix

\section{Omitted Proofs on Centrality: Bounded Differences and $2\rightarrow 1$ Norm}
\label[section]{sec:omitted-centrality}

We turn to the proofs of \cref{lem:2-to-1,lem:bdd-differences}, starting with the former.
The proof uses ideas from the empirical process literature. Lugosi and Mendelson prove a similar statement in the course of proving \cite[Lemma 1]{LM18}.
We will need the Ledoux-Talagrand contraction lemma:

\begin{lemma}[Ledoux-Talagrand Contraction, as stated in \cite{LM18}]
\label[lemma]{lem:contraction}
  Let $X_1,\ldots,X_n$ be i.i.d. random vectors taking values in $\R^d$.
  Let $\cF$ be a class of real-valued functions defined on $\R^d$.
  Let $\sigma_1,\ldots,\sigma_n$ be independent Rademacher random variables, independent of the $X_i$.
  If $\phi \, : \, \R \rightarrow \R$ is a function with $\phi(0) = 0$ and Lipschitz constant $L$, then
  \[
  \E \sup_{f \in \cF} \sum_{i \leq n} \sigma_i \phi(f(X_i)) \leq L \cdot \E \sup_{f \in \cF} \sum_{i \leq n} \sigma_i f(X_i)\mper
  \]
\end{lemma}

\begin{proof}[Proof of \cref{lem:2-to-1}]
  First, for any unit $v \in \R^d$,
  \[
    \E_Z \Abs{Z_i,v} \leq \sqrt{\E_Z \iprod{Z,v}^2} = \sqrt{\iprod{v, \Sigma v}} \leq \sqrt{\|\Sigma\|}\mper
  \]
  Let $\bZ = Z_1,\ldots,Z_k$
  We have that
  \begin{align*}
  \E \|M\|_{2 \rightarrow 1} & = \E_{\bZ} \sup_{\|v\|=1} \sum_{i \leq k} \Abs{\iprod{Z_i,v}}\\
  & \leq \Brac{\E_{\bZ} \sup_{\|v\|=1}\Paren{ \sum_{i \leq k} \Abs{\iprod{Z_i,v}} - \E\Abs{ \iprod{Z_i,v}}}} + k \cdot \sup_{\|v\|=1} \E_Z \Abs{\iprod{Z,v}}\\
  & \leq  \Brac{\E_{\bZ} \sup_{\|v\|=1}\Paren{ \sum_{i \leq k} \Abs{\iprod{Z_i,v}} - \E \Abs{\iprod{Z_i,v}}}} + k \cdot \sqrt{\|\Sigma\|} \mper
  \end{align*}
  Thus it will suffice to show that $\E_{\bZ} \sup_{\|v\|=1}\Paren{ \sum_{i \leq k} \Abs{\iprod{Z_i,v}} - \E \Abs{\iprod{Z_i,v}}} \leq 2 \sqrt{k \Tr \Sigma}$.

  We use a symmetrization argument.
  Let $Z_i'$ be an independent copy of $Z_i$ and let $\bZ' = Z_1',\ldots,Z_k'$.
  Let $\sigma_1,\ldots,\sigma_k \sim \{\pm 1\}$ be i.i.d. random signs.
  Then by standard symmetrization,
  \begin{align*}
  \E_{\bZ} \sup_{\|v\|=1} & \Paren{ \sum_{i \leq k}  \Abs{\iprod{Z_i,v}} - \E \Abs{\iprod{Z_i,v}}}\\
  & \leq \E_{\bZ,\bZ',\sigma} \sup_{\|v\|=1} \Paren{\sum_{i \leq k} \sigma_i(| \iprod{Z_i,v}| - |\iprod{Z_i',v}|)}\\
  & \leq 2 \E_{\bZ,\sigma} \sup_{\|v\|=1} \sum_{i \leq k} \sigma_i |\iprod{Z_i,v}|
  \end{align*}
  where we have used triangle inequality for the last step.
  Now since the absolute value function is $1$-Lipschitz, by \cref{lem:contraction} this is at most
  \[
  2 \E_{\bZ, \sigma} \sup_{\|v\|=1} \sum_{i \leq k} \sigma_i \iprod{Z_i,v} = 2 \E_{\bZ,\sigma} \Norm{ \sum_{i \leq k} \sigma_i Z_i} \mper
  \]
  Squaring and expanding this norm,
  \begin{align*}
  2 \cdot \E_{\bZ,\sigma} \Norm{\sum_{i \leq k} \sigma_i Z_i} & \leq 2 \cdot \Paren{ \E_{\bZ,\sigma} \sum_{ij \leq k} \sigma_i \sigma_j \iprod{Z_i,Z_j} }^{1/2}\\
  & = 2 \cdot \Paren{ k \cdot \Tr \Sigma }^{1/2}\mcom
  \end{align*}
  which concludes the proof.
\end{proof}

Next we turn to the proof of \cref{lem:bdd-differences}.

\begin{proof}[Proof of \cref{lem:bdd-differences}]
  Without loss of generality we may assume that $i=k$.
  By symmetry, it is enough to show that
  \[
  SDP(Z_1,\ldots,Z_k,x,r) \leq SDP(Z_1,\ldots,Z_k',x,r) + \frac 1k\mper
  \]
  Consider a feasible solution $Y(B,W,U,b,r)$ to $SDP(Z_1,\ldots,Z_k,x,r)$.
  By setting the $(k+1)$-st row and column of $Y$ (containing $B_{kk}, b_k$) to $0$, we obtain $Y(B',W',U,b',r)$ which is feasible for $SDP(Z_1,\ldots,Z_k',x,r)$.
  Since $b_k \leq \sqrt{B_{kk}} \leq 1$ by positivity, the objective value of $Y(B',W',U,b',u)$ is at most $1/k$ less than that of $Y(B,W,U,b,u)$.
\end{proof}

\section{\textsc{median-sdp}}
\label[section]{sec:median-sdp}

In this section we prove \cref{lem:sos-median-proof}.
The proof is technical, but a useful intuition is that it casts as a series of SoS-provable polynomial inequalities the following simple argument about $(r,1/10)$-central points.
\begin{quote}
If $x,y$ are both $(r,1/10)$-central with respect to $Z_1,\ldots,Z_k$, then in the direction $u = (x-y)/\|x-y\|$ there exists a shared inlier $Z_i$ such that $|\iprod{Z_i,u} - \iprod{x,u}| \leq r$ and $|\iprod{Z_i,u} - \iprod{y,u}| \leq r$.
Therefore $\|x-y\| \leq |\iprod{x,u} - \iprod{y,u}| \leq 2r$.
\end{quote}

The most challenging part of this argument to mimic in the SoS proof system is the existence of the shared inlier.
The difficulty boils down to this: \emph{the $0/1$ indicator for $\iprod{Z_i-x,v} > r$ is not a polynomial in $x,v$}.
To get around this issue, we carefully use the auxiliary variables $a_{i,t}$ in $\cA$; from a high level they allow us to construct a proxy for that $0/1$ indicator which is a polynomial function in the variables of $\cA$.

For the remainder of this section, let $\mu$ be a certifiably $(r,1/10)$-central point, with nonnegative $\alpha_1',\ldots,\alpha_k',\beta_1',\ldots,\beta_k',\gamma'$ and a degree-$2$ SoS polynomial $\sigma'$ comprising its witness as in \cref{lem:witness}.
To prove \cref{lem:sos-median-proof} we need to assemble some SoS proofs.
Note that $\alpha',\beta'$ are in $\R$, while $\alpha,\beta$ are indeterminates involved in the polynomial system $\cA$: all the SoS proofs which follow are in variables $\alpha,\beta,\sigma,x$, while $\alpha',\beta',\gamma',\sigma',\mu$, and $\gamma$ appear in the \emph{coefficients} of the polynomials involved.

\begin{definition}
  Let $a_i(a_{i,0},\ldots,a_{i,t})$ be the polynomial $a_i = \tfrac{a_{i,0}}{\alpha_i'} + \sum_{1 \leq t \leq \log C/c+1} \frac{a_{i,t}}{2^{t} \cdot c + \alpha_i'}$.
  Note that the values of $2^{t} \cdot c$ range from $2c$ to $2C$.
\end{definition}

It is useful to think of the polynomial $a_i$ as an approximation to $\tfrac{1}{\alpha_i + \alpha_i'}$.

We prove the following Lemmas in \cref{sec:remaining-sos} by elementary means.

\begin{lemma}\label[lemma]{lem:b-bounds}
Let $b_i(\alpha,a) = \alpha_i' a_i$ and $b_i'(\alpha,a) = \alpha_i a_i$.
Then
\[
\cA \proves{4} b_i^2 \leq 1, (b_i')^2 \leq 1, b_i \leq 1, b_i' \leq 1
\]
and
\[
\cA \proves{4} b_i + b_i' \geq \frac 12\mper
\]
\end{lemma}

\begin{lemma}\label[lemma]{lem:alpha-bounds}
$\cA \proves{8} 0.1 \cdot \frac kr \leq \sum_{i \leq k} \alpha_i \alpha_i' a_i \leq 0.6 \cdot \frac kr$
\end{lemma}

\begin{lemma}\label[lemma]{lem:sigma-bounds}
  Let $b,v$ be any polynomials of degree at most $4$ in $x,\alpha,\beta,\sigma,a$.
  Then
  \[
  \cA \proves{8} \sigma(b,v) \geq 0\mper
  \]
\end{lemma}

\begin{proof}[Proof of \cref{lem:sos-median-proof}]
  We are going to evaluate the certificates $\alpha,\beta,\gamma,\sigma$ and $\alpha',\beta',\gamma',\sigma'$ at some carefully chosen $b,v,b',v'$.
  Let $\Delta = \sqrt{\pE \|\mu -x\|^2}$.
  Let $v,v'$ be the following polynomials in $x$:
  \[
  v = \frac{\mu - x}{\Delta} \text{ and } v' = -v = \frac{x - \mu}{\Delta}\mper
  \]
  Let $b,b'$ be the following polynomials in $\alpha$ and $a$
  \[
  b_i = \alpha_i' a_i \text{ and } b_i' = \alpha_i a_i\mper
  \]
  Then if we evaluate $x$'s certificate at $b,v$ we get
  \begin{align*}
  \cA \proves{8} \frac k {10} - \sum_{i=1}^k b_i & = \sum_{i=1}^k \alpha_i b_i (\iprod{Z_i -x, (\mu - x)/\Delta} - r)
  + \sum_{i=1}^k \beta_i (1-b_i^2) \\
  & + \gamma (1 - \|\mu-x\|^2/\Delta^2)
  + \sigma(b,v)\mper
  \end{align*}
  (This implicitly uses all the linear equalities $\cA_{\text{sos}}$.)
  Doing the same for $\alpha',\beta',\gamma',\sigma'$ evaluated at $b',v'$ and adding the result to the above,
  \begin{align*}
  \cA \proves{8} & \frac k {5} - \sum_{i=1}^k (b_i + b_i') \\
  & = \sum_{i=1}^k \alpha_i \alpha_i' a_i (\iprod{Z_i -x, (\mu - x)/\Delta} + \iprod{Z_i-\mu, (x-\mu)/\Delta} - 2r) \\
  & + \sum_{i=1}^k \beta_i (1-b_i^2) + \beta_i' (1-(b_i')^2) \\
  & + \gamma (1 - \|\mu-x\|^2/\Delta^2) + \gamma' (1-\|\mu-x\|^2/\Delta^2) \\
  & + \sigma(b,v) + \sigma'(b',v') \mper
  \end{align*}
  Notice that $\iprod{Z_i-x,\mu-x} + \iprod{Z_i-x,x-\mu}$ rearranges to $\|\mu - x\|^2$.
  Also using $\cA \proves{8} \beta_i(1-b_i^2), \beta_i' (1-(b_i')^2) \geq 0$ (by \cref{lem:b-bounds}) and $\cA \proves{8} \sigma(b,v), \sigma'(b',v') \geq 0$ (by \cref{lem:sigma-bounds}),
  \begin{align*}
  \cA \proves{8} \frac k 5 - \sum_{i=1}^k (b_i + b_i') & \geq \sum_{i=1}^k \alpha_i \alpha_i' a_i (\|\mu - x\|^2 / \Delta - 2r) \\
  & + (\gamma + \gamma')(1 - \|\mu-x\|^2/\Delta^2) \mper
  \end{align*}
  Now using \cref{lem:b-bounds}, which says $\cA \proves{8} \sum_{i=1}^k b_i + b_i' \geq k/2$,
  \[
  \cA \proves{8} 0 \geq \sum_{i=1}^k \alpha_i \alpha_i' a_i(\|\mu - x\|^2/\Delta - 2r) + (\gamma + \gamma')(1 - \|\mu - x\|^2/\Delta^2)\mper
  \]
  Since $\pE$ satisfies $\cA$,
  \[
  \frac 1 {\Delta} \pE \sum_{i=1}^k \alpha_i \alpha_i' a_i \|\mu - x\|^2 \leq 2r \pE \sum_{i=1}^k \alpha_i \alpha_i' a_i - \pE (\gamma + \gamma')(1 - \|\mu -x\|^2/\Delta^2)\mper
  \]
  By definition $\pE \|\mu -x\|^2 = \Delta^2$ and $\pE (\gamma + \gamma') \|\mu - x\|^2 = (\gamma + \gamma') \pE \|\mu - x\|^2$ because $\gamma,\gamma'$ are numbers in $\R$ rather than indeterminates.
  So $\pE (\gamma + \gamma') (1 - \|\mu - x\|^2/\Delta^2) = 0$, and hence
  \[
    \frac 1 {\Delta} \pE \sum_{i=1}^k \alpha_i \alpha_i' a_i \|\mu - x\|^2 \leq 2r \pE \sum_{i=1}^k \alpha_i \alpha_i' a_i\mper
  \]
  By \cref{lem:alpha-bounds},
  \[
  \frac 1 {\Delta} \pE \|\mu - x\|^2 \leq 10 \cdot \frac rk \cdot \frac 1{\Delta} \cdot \pE \sum_{i=1}^k \alpha_i \alpha_i' a_i \|\mu - x\|^2
  \]
  and $\pE \sum_{i=1}^k \alpha_i \alpha_i' a_i \leq 0.6 \cdot \frac k r$, so putting it all together $\frac 1 {\Delta} \pE \|\mu - x\|^2 \leq O(r)$.
  By the definition of $\Delta = \sqrt{\pE \|\mu - x\|^2}$, we get $\pE \|\mu -x\|^2 = O(r^2)$.
\end{proof}

\section{Well-Conditioned Witnesses}
In this section we prove \cref{lem:nice-witness}.
We have to establish two separate facts.
First, we show that $\gamma$ can be taken to be in the set $\{0,1/100,2/200,\ldots,k\}$, and second, that $\alpha_1,\ldots,\alpha_k$ can be taken either to be $0$ or in $[c,C]$ for some numbers $C > c > 0$.
These properties correspond to the following two lemmas, from which \cref{lem:nice-witness} follows immediately by first applying \cref{lem:fix-alpha} and then \cref{lem:fix-gamma}.

\begin{lemma}[Obtaining nice $\gamma$]\label[lemma]{lem:fix-gamma}
  Let $Z_1,\ldots,Z_k,\mu \in \R^d$.
  Suppose there exist nonnegative $\alpha_1,\ldots,\alpha_k,\beta_1,\ldots,\beta_k,\gamma$ and a degree $2$ SoS polynomial $\sigma(b,v)$ such that the identity
  \begin{align*}
    C - \sum_{i=1}^k b_i & = \sum_{i=1}^k \alpha_i b_i (\iprod{Z_i - \mu,v} - r)
     + \sum_{i=1}^k \beta_i (1-b_i^2)\\
     & + \gamma (1-\|v\|^2)
     + \sigma (b,v)\mper
  \end{align*}
  holds in variables $b_1,\ldots,b_k,v_1,\ldots,v_d$.
  Then there are $\gamma',\sigma'$ with $\gamma' \in \{0,1/100,2/100,\ldots,\lceil C \rceil \}$ and $\sigma'$ a degree-$2$ SoS polynomial such that
  \begin{align*}
    C + 1/100 - \sum_{i=1}^k b_i & = \sum_{i=1}^k \alpha_i b_i (\iprod{Z_i - \mu,v} - r)
     + \sum_{i=1}^k \beta_i (1-b_i^2) \\
    & + \gamma' (1-\|v\|^2)
    + \sigma' (b,v)\mper
  \end{align*}
\end{lemma}
\begin{proof}
  First by taking each $b_i$ and $v$ to be $0$ and evaluating the hypothesized the polynomial identity, we find that $0 \leq \gamma \leq C$.
  (Here we used nonnegativity of $\beta_1,\ldots,\beta_k$).
  Replacing $\gamma$ with the next greatest number $\gamma'$ of the form $i/100$ for $i$ an integer incurs an additive error $(\gamma - \gamma') + (\gamma' - \gamma)\|v\|^2$.
  Moving $\gamma - \gamma'$ to the left-hand side replaces $C$ with $C + \gamma' - \gamma < C+1/100$.
  The polynomial $(\gamma' - \gamma) \|v\|^2$ is a degree $2$ sum of squares, so it can be added to $\sigma$ to obtain $\sigma'$.
\end{proof}

\begin{lemma}[Obtaining nice $\alpha_i$'s]\label[lemma]{lem:fix-alpha}
  Suppose that $Z_1,\ldots,Z_k,\mu \in \R^d$ have the property that $\|Z_i - \mu\| > C$ for at most $k'$ indices $i \in [k]$.
  And suppose that there exist nonnegative numbers $\alpha_1,\ldots,\alpha_k$, $\beta_1,\ldots,\beta_k$, and $\gamma$, and $p \in [0,1]$, and a degree-2 SoS polynomial $\sigma$ in variables $b_1,\ldots,b_k, v_1,\ldots,v_d$ such that as polynomials in $b,v$,
  \begin{align*}
  pk - \sum_{i =1}^k b_i & = \sum_{i=1}^k \alpha_i b_i ( \iprod{Z_i-\mu,v} - r) + \sum_{i=1}^k \beta_i(1-b_i^2) \\
  & + \gamma(1-\|v\|^2) + \sigma(b,v)\mper
  \end{align*}
  Then there are nonnegative $\alpha_1',\ldots,\alpha_k',\beta_1',\ldots,\beta_k',\gamma'$ and a degree-$2$ SoS polynomial $\sigma'$ such that $\alpha_i' = 0$ if $\|Z_i - \mu\| > C$ and otherwise $4 k/r \geq \alpha_i' \geq \min(1/C,1/r)/100$ and
 \begin{align*}
    pk + k' + \frac k {25} - \sum_{i=1}^k b_i & = \sum_{i=1}^k \alpha_i' b_i (\iprod{Z_i - \mu,v} - r)
    + \sum_{i=1}^k \beta_i' (1-b_i^2) \\
    & + \gamma'(1-\|v\|^2)
    + \sigma'(b,v)\mper
 \end{align*}

\end{lemma}

To prove \cref{lem:fix-alpha}, we are going to use an SoS version of the Cauchy-Schwarz inequality.
(The proof has appeared many times before.)

\begin{lemma}[SoS Cauchy-Schwarz (folklore)]
\label[lemma]{lem:sos-cs}
  Let $x_1,\ldots,x_n, y_1,\ldots,y_n$ be indeterminates.
  Then $2 \iprod{x,y} \preceq \|x\|^2 + \|y\|^2$.
\end{lemma}
\begin{proof}
  By expanding, $\|x - y\|^2 = \|x\|^2 + \|y\|^2 - 2\iprod{x,y}$.
\end{proof}

\begin{proof}[Proof of \cref{lem:fix-alpha}]
  Let $S \subseteq [k]$ with $|S| = k - k'$ be those indices where $\|Z_i - \mu\| \leq C$.
  Then by setting $b_i = 0$ for $i \notin S$, the following polynomial identity holds
  \begin{align*}
    pk - \sum_{i \in S} b_i & = \sum_{i \in S} \alpha_i b_i (\iprod{Z_i - \mu,v} - r) + \sum_{i \in S} \beta_i (1-b_i^2) + \sum_{i \notin S} \beta_i \\
    & + \gamma(1-\|v\|^2) + \tau(b,v)
  \end{align*}
  where $\tau(b,v)$ is the SoS polynomial obtained by partially evaluating $\sigma$ with $b_i = 0$ when $i \notin S$.
  We will make a series of modifications to this polynomial identity.

  First, we need to replace $\sum_{i \in S} b_i $ with $\sum_{i =1}^k b_i$.
  We add the polynomial $\sum_{i \notin S} (1-b_i)^2 /2$ to both sides, to get
  \begin{align*}
  pk - \sum_{i \leq k} b_i  + k'/2 + \sum_{i \notin S} b_i^2/2 & = \sum_{i \in S} \alpha_i b_i (\iprod{Z_i - \mu,v} - r) + \sum_{i \in S} \beta_i (1-b_i^2) \\
    & + \gamma(1-\|v\|^2) + \tau'(b,v)
  \end{align*}
  where $\tau'$ is another sum of squares.
  Then moving $\sum_{i \notin S} b_i^2/2$ to the other side and adding $k'/2$ to both sides, we find some nonnegative $\beta_i'$ such that
  \begin{align*}
  pk + k' - \sum_{i \leq k} b_i & = \sum_{i \in S} \alpha_i b_i (\iprod{Z_i - \mu,v} - r) + \sum_{i \leq k} \beta_i' (1-b_i^2) \\
    & + \gamma(1-\|v\|^2) + \tau'(b,v)\mper
  \end{align*}

  At this stage we have obtained a proof where $\alpha_i = 0$ if $\|Z_i - \mu\| > C$.
  Next we would like to ensure that the remaining $\alpha_i$'s are not too small.

  Let $T \subseteq S$ be those indices $i$ such that $\alpha_i < \min(1/100C,1/100r)$.
  By Cauchy-Schwarz (\cref{lem:sos-cs}),
  \[
  \alpha_i b_i \iprod{Z_i - \mu,v} = \iprod{(10\alpha_i b_i)(Z_i - \mu), v/10} \preceq 100 \alpha_i^2 b_i^2 \|Z_i - \mu\|^2 + \|v\|^2 /100\mper
  \]
  We apply this for $i \in T$ to conclude that
  \begin{align*}
    \sum_{i \in T } \alpha_i r + & pk + k' - \sum_{i \leq k} b_i\\
    & = \sum_{i \in S \setminus T} \alpha_i b_i (\iprod{Z_i - \mu,v} - r) + \sum_{i \in S} \beta_i' (1-b_i^2) + \gamma(1-\|v\|^2)\\
    & - \sum_{i \in T} 100 \alpha_i^2 b_i^2 \|Z_i - \mu\|^2
    - \frac{k}{100} \|v\|^2
    + \psi(b,v)
  \end{align*}
  for yet another degree $2$ SoS polynomial $\psi$.
  Finally, note that $100 \alpha_i^2 \|Z_i - \mu\|^2 \leq 1/100$ by hypothesis.
  We will use this to absorb $100 \alpha_i^2 b_i^2 \|Z_i-\mu\|^2$ into the $(1-b_i^2)$ terms.
  By adding $k/50$ to both sides, for some nonnegative $\beta_i'', \gamma'$ we get
   \begin{align*}
    \sum_{i \in T } & \alpha_i r + \frac{k}{50} + pk + k' - \sum_{i=1}^k b_i\\
    & = \sum_{i \in S \setminus T} \alpha_i b_i (\iprod{Z_i - \mu,v} - r)
    + \sum_{i \in S} \beta_i'' (1-b_i^2)
    + \gamma'(1-\|v\|^2)
    + \psi(b,v)\mper
  \end{align*}
  Again by definition of the set $T$, $\sum_{i \in T} \alpha_i r \leq k/100$.
  So we get
  \begin{align*}
  \frac k{25} + pk + k' - \sum_{i=1}^k b_i & = \sum_{i =1}^k \alpha_i' b_i(\iprod{Z_i-\mu,v} - r) + \sum_{i \in S} \beta_i''(1-b_i^2)\\
  & + \gamma'(1-\|v\|^2) + \chi(b,v)
  \end{align*}
  where $\alpha_i' = 0$ or $\alpha_i' > \min(1/C,1/r)/100$ and $\chi$ is another degree-$2$ SoS polynomial.

  It only remains to ensure that each $\alpha_i'$ is not too large.
  If we set $b_i=-1$ for every $i$ and take $v =0$, we get
  \[
  pk + k' + \frac k {25} + k = r \sum_{i=1}^k \alpha_i' + \chi(b,v) \geq r \sum_{i=1}^k \alpha_i'\mper
  \]
  Hence $4k/r \geq \alpha_i'$ for all $i \in [k]$.
\end{proof}

\section{Remaining SoS proofs}
\label[section]{sec:remaining-sos}

We turn to the proofs of \cref{lem:b-bounds,lem:alpha-bounds,lem:sigma-bounds}.

\begin{proof}[Proof of \cref{lem:b-bounds}]
  Starting with the first statement, since $\cA$ includes $a_{i,t}a_{i,t'} = 0$ if $t \neq t'$,
  \[
  \cA \proves{2} b_i^2 = \sum_{1 \leq t \leq \log(C/c)+1} \frac{(\alpha_i')^2 a_{i,t}^2}{(2^{t}c + \alpha_i')^2} + \frac{(\alpha_i')^2 a_{i,0}^2}{(\alpha_i')^2}\mper
  \]
  Since $a_{i,t}^2$ is a square and $2^{t} c \geq 0$,
  \[
  \cA \proves{2} b_i^2 \leq \sum_t \frac{(\alpha_i')^2 a_{i,t}^2}{(\alpha_i')^2} = 1
  \]
  where in the last step we used $\sum_t a_{i,t} = \sum_t a_{i,t}^2 = 1$.

  Next we show $\cA \proves{4} (b_i')^2 \leq 1$.
  Proceeding similarly as before,
  \[
  \cA \proves{4} (b_i')^2 = \sum_{1 \leq t \leq \log(C/c)} \frac{\alpha_i^2 a_{i,t}^2}{(2^{t}c + \alpha_i')^2} + \frac{\alpha_i^2 a_{i,0}^2}{(\alpha_i')^2} \leq \sum_t \frac{\alpha_i^2 a_{i,t}^2}{(2^{t}c)^2}\mper
  \]
  Using $\alpha_i a_{i,t} \leq 2^{t}c a_{i,t}$, we get
  \[
  \cA \proves{4} (b_i')^2 \leq \sum_{t} a_{i,t}^2 = 1\mper
  \]

  The proofs of $\cA \proves{4} b_i, b_i' \leq 1$ are similar, so we move on to the last statement.
  \begin{align*}
  \cA \proves{2} b_i + b_i' & = \sum_{1 \leq t \leq \log(C/c)+1} \frac{a_{i,t} (\alpha_i + \alpha_i')}{2^{t} c + \alpha_i'} + \frac{a_{i,0}(\alpha_i + \alpha_i')}{\alpha_i'}\\
  & \geq \sum_t \frac{a_{i,t}(2^{t-1} c + \alpha_i')}{2^{t}c + \alpha_i'} + a_{i,0}\mper
  \end{align*}
  Since $a_{i,t} =a_{i,t}^2$ we get
  \[
  \cA \proves{2} \frac{a_{i,t}(2^{t-1} c + \alpha_i')}{2^{t}c + \alpha_i'} \geq \frac{a_{i,t}}{2}
  \]
  and hence $\cA \proves{2} b_i + b_i' \geq \tfrac 12 \sum_t a_{i,t} = \tfrac 12$.
\end{proof}

\begin{proof}[Proof of \cref{lem:alpha-bounds}]
  We start with the lower bound $\cA \proves{8} \sum_{i=1}^k \alpha_i \alpha_i' a_i \geq 0.1k/r$.
  Let $b_i(\alpha,a) = \alpha_i' a_i$ and $b_i'(\alpha,a) = \alpha_i a_i$.
  Let $v,v' \sim \cN(0,\Id/d)$.
  Using the certificate for $\mu$ evaluated at $b',v'$ and averaging over $v'$, using $\E v' = 0$ and $\E \|v'\|^2 = 1$,
  \[
  \cA \proves{8} \frac k {10} - \sum_{i=1}^k \alpha_i a_i = \sum_{i=1}^k \beta_i'(1-(b_i')^2) - r \cdot \sum_{i=1}^k \alpha_i \alpha_i' a_i + \E_v \sigma'(b',v')
  \]
  By \cref{lem:b-bounds}, $\cA \proves{8} \sum_{i=1}^k \beta_i'(1-(b_i')^2) \geq 0$.
  And $\sigma'$ is a sum of squares, so the above rearranges to
  \[
  \cA \proves{8} \frac k {10} - \sum_{i=1}^k \alpha_i a_i + r \sum_{i=1}^k \alpha_i \alpha_i' a_i \geq 0\mper
  \]
  Using the same argument on $x$'s certificate evaluated at $b,v$ and averaged over $v$,
  \[
  \cA \proves{8} \frac k {10} - \sum_{i=1}^k \alpha_i' a_i + r \sum_{i=1}^k \alpha_i \alpha_i' a_i \geq 0\mper
  \]
  By adding these together, we get
  \[
  \cA \proves{8} \frac k {5} - \sum_{i=1}^k a_i (\alpha_i + \alpha_i') + 2 r \sum_{i=1}^k \alpha_i \alpha_i' a_i \geq 0\mper
  \]
  Since $\cA \proves{8} a_i(\alpha_i + \alpha_i') = b_i + b_i' \geq 1/2$ by \cref{lem:b-bounds}, we get
  \[
  \cA \proves{8} 2r \sum_{i=1}^k \alpha_i \alpha_i' a_i \geq \frac k2 - \frac k5 \geq \frac k5
  \]
  which rearranges to $\cA \proves{8} \sum_{i=1}^k \alpha_i \alpha_i' a_i \geq k/(10r)$.

  Turning now to the upper bound, let us redefine $b_i = - \alpha_i' a_i$ and $b_i' = -\alpha_i a_i$.
  Note that this does not change $b_i^2, (b_i')^2$.
  Using the same arguments, now we obtain
  \[
  \cA \proves{8} \frac k {10} + \sum_{i=1}^k \alpha_i a_i - r \sum_{i=1}^k \alpha_i \alpha_i' a_i \geq 0
  \]
  and
  \[
  \cA \proves{8} \frac k {10} + \sum_{i=1}^k \alpha_i' a_i - r \sum_{i=1}^k \alpha_i \alpha_i' a_i \geq 0\mper
  \]
  Adding these together and rearranging,
  \[
  \cA \proves{8} \frac k {5} + \sum_{i=1}^k a_i (\alpha_i + \alpha_i') \geq 2r \sum_{i=1}^k \alpha_i \alpha_i' a_i\mper
  \]
  By \cref{lem:b-bounds}, we know $\cA \proves{8} \sum_{i=1}^k a_i (\alpha_i + \alpha_i') \leq k$, which finishes the proof.
\end{proof}

\begin{proof}[Proof of \cref{lem:sigma-bounds}]
  Follows by definition, since $\sigma(b,v)$ is a sum of squares of degree $8$ in $\sigma_{ij}, \alpha,\beta,x,a$.
\end{proof}

\section{Removing dependence on $\Sigma$}
\label[section]{sec:avoid-assumptions}

When we proved Theorem 1.2, we made the additional assumption that the algorithm is given access to number $r,C > 0$ in addition to the samples $X_1,\ldots,X_n$.
We describe in this section how the dependence on $r,C$ can be avoided.

First of all, we note that although $C$ is an independent parameter in \cref{lem:sos-median}, in fact the proof shows that it suffices to choose $C = C' k$ for a universal constant $k$.
This just leaves the parameter $r$.
The main idea is to use binary search to adaptively choose $r$.
The crucial observation is that the proof of \cref{lem:sos-median} actually proves the following stronger statement.

\begin{lemma}[Refined version of \cref{lem:sos-median}]\label[lemma]{lem:sos-median-appendix}
  For every $d,k \in \N$ and $C,r > 0$ there is an algorithm \textsc{median-sdp} which runs in time $(dk \log C)^{O(1)}$ and has the following guarantees.
  Let $Z_1,\ldots,Z_k \in \R^d$.
  Suppose that $\mu \in \R^d$ is certifiably $(r',1/100)$-central with respect to $Z_1,\ldots,Z_k$.
  And, suppose that at most $k/100$ of the vectors $Z_1,\ldots,Z_k$ have $\|Z_i - \mu\| > Cr'$.
  Then:
  \begin{itemize}
    \item If $2r' \geq r \geq r'/2$, \textsc{median-sdp} returns a point $x$ such that $\|\mu - x\| = O(r)$.
    \item Otherwise, \textsc{median-sdp} either returns $x$ such that $\|\mu - x\| \leq O(r + r')$ or outputs \textsc{reject}.
  \end{itemize}
\end{lemma}

The proof of \cref{lem:sos-median-appendix} follows exactly the proof of \cref{lem:sos-median}, with the following additional observations.
First, for any $r$, if there exists a pseudodistribution satisfying $\cA(Z_1,\ldots,Z_k,r,1/200Cr,8k/r,\gamma)$ for some $\gamma$ and if $\mu$ is certifiably $(r',1/100)$-central, then \textsc{median-sdp} returns $x$ with $\|x - \mu\| \leq O(r+r')$.
Second, so long as $r \in [r'/2,2r']$, then $\cA(Z_1,\ldots,Z_k,r,1/200Cr,8k/r,\gamma)$ will be feasible for some choice of $\gamma$, so such a pseudodistribution will exist.

Thus, the adaptive algorithm will use binary search to choose the smallest $r$ such that \textsc{median-sdp} does not output \textsc{reject}.
We just have to ensure that the range of potential values for $r$ to search over is not too large.

Fix an underlying random variable $X$ with covariance $\Sigma$.
Let $r^* = \sqrt{\Tr \Sigma / n} + \sqrt{\|\Sigma\| \log(1/\delta)/n}$.
We may assume that the radius $r_0$ of the minimum-size ball containing at least $0.8k$ of the bucketed means $Z_i$ is in the range $[cr^*, C \log(1/\delta) r]$ for some constants $c,C$.
If $r_0 < c r^*$ then the simple median algorithm (see Section 1 of the main paper) finds an estimator with the guarantees of Theorem 1.2.
And by our analysis of that algorithm in Section 1 of the main paper, $r_0 > C \sqrt{\Tr \Sigma \log(1/\delta)}$ only with probability $\delta$.

The adaptive algorithm can begin by computing $r_0$ (or, more precisely, by inspecting only balls centered at $Z_i$'s, a number $r_0'$ such that $r_0' \in [r_0, 4r_0]$).
Then by conducting binary search over values of $r$ in the range $[c r_0 / k, Cr_0]$ for some (other) constants $c,C$, it will find $r \in [r^*/2, 2r^*]$.
\cref{lem:sos-median-appendix} concludes the argument.


\end{document}